\documentclass[11pt, final]{article}
\usepackage[top=1.1in, bottom=1.1in, left=1.1in, right=1.1in]{geometry}

\usepackage{enumitem}
\setlist[description]{leftmargin=3.5mm}

\RequirePackage{amsthm}
\theoremstyle{plain} 
\newtheorem{theorem}{Theorem}[section]

\newtheorem{lemma}      [theorem]{Lemma}

\usepackage{amsmath,amssymb} 

\usepackage{amsfonts}
\usepackage{graphicx}
\usepackage{epstopdf}
\usepackage{algorithmic}

\ifpdf
  \DeclareGraphicsExtensions{.eps,.pdf,.png,.jpg}
\else
  \DeclareGraphicsExtensions{.eps}
\fi

\usepackage{hyperref}
\hypersetup{
     colorlinks   = true,
     citecolor    = blue,
     linkcolor    = blue,
     urlcolor     = blue
}

\usepackage{tikz,pgfplots}

\usetikzlibrary{shapes}
\usepackage{graphicx}

\usepackage{algorithm,algorithmic}

\newcommand{\rottext}[1]{\rotatebox{90}{\hbox to 20mm{\hss #1\hss}}}
\newcommand\norm[1]{\left\lVert#1\right\rVert}
\usepackage{datetime}

\newcommand{\R}{\mathbb{R}}

\renewcommand{\norm}[2][]{\left\Vert#2\right\Vert_{#1}}

\DeclareMathOperator*{\argmin}{arg\,min}                   
\renewcommand{\t} {^{\top}}                                

\renewcommand{\phi}{\mathbf{\varphi}}

\newcommand{\bfzero}{{\bf0}}

\newcommand{\bfd}{\mathbf{d}}

\newcommand{\bfA}{\mathbf{A}}
\newcommand{\bfC}{\mathbf{C}}

\newcommand{\bfU}{\mathbf{U}}

\newcommand{\bfR}{\mathbf{R}}
\newcommand{\bfI}{\mathbf{I}}
\newcommand{\bfD}{\mathbf{D}}

\newcommand{\bfb}{\mathbf{b}}

\newcommand{\bfB}{\mathbf{B}}

\newcommand{\bfx}{\mathbf{x}}

\newcommand{\bfe}{\mathbf{e}}
\newcommand{\bfu}{\mathbf{u}}
\newcommand{\bfy}{\mathbf{y}}

\newcommand{\bfL}{\mathbf{L}}

\newcommand{\bfz}{\mathbf{z}}

\newcommand{\bfr}{\mathbf{r}}
\newcommand{\bfs}{\mathbf{s}}
\newcommand{\bfv}{\mathbf{v}}
\newcommand{\bfV}{\mathbf{V}}

\newcommand{\bfQ}{\mathbf{Q}}

\newcommand{\bfS}{\mathbf{S}}

\newcommand{\bfY}{\mathbf{Y}}

\newcommand{\trace}{{\mathop{\mathrm{tr}}}}

\newcommand{\calN}{\mathcal{N}}
\newcommand{\calO}{\mathcal{O}}

\newcommand{\calR}{\mathcal{R}}

\newcommand{\bfepsilon}{{\boldsymbol{\epsilon}}}

\newcommand{\bfmu}{{\boldsymbol{\mu}}}

\newcommand{\bfxi}{{\boldsymbol{\xi}}}

\newcommand{\bbE}{\mathbb{E}}

\newcommand{\bbR}{\mathbb{R}}

\usetikzlibrary{arrows}

\newlength\iwidth
\newlength\iheight

\usepackage{subcaption}
\usepackage{caption}
\usepackage{tabularx}
\usepackage{graphicx}
\usepackage{array}
\usepackage{longtable}
\usepackage{verbatim}
\newcommand{\true}{{\rm true}}

\newdimen\iwidth
\newdimen\iheight

\theoremstyle{plain}

\newcommand{\TheTitle}{Hybrid Projection Methods for Large-scale Inverse Problems with Mixed Gaussian Priors}

\title{{\TheTitle}}

\author{
Taewon Cho\thanks{Department of Mathematics, Virginia Tech, Blacksburg, VA, USA
  (taewon88@vt.edu).}
 \and
  Julianne Chung\thanks{Department of Mathematics, Computational Modeling and Data Analytics Division, Academy of Integrated Science, Virginia Tech, Blacksburg, VA, USA
    (jmchung@vt.edu, http://www.math.vt.edu/people/jmchung/).}
  \and
Jiahua Jiang\thanks{Department of Mathematics, Virginia Tech, Blacksburg, VA, USA
  (jiahua@vt.edu).}
}

\usepackage{amsopn}

\definecolor{darkcyan}{rgb}{0.0, 0.55, 0.55}

\usepackage[normalem]{ulem}
\pgfplotsset{compat=1.14}

\begin{document}
\maketitle

\begin{abstract}
When solving ill-posed inverse problems, a good choice of the prior is critical for the computation of a reasonable solution.  A common approach is to include a Gaussian prior, which is defined by a mean vector and a symmetric and positive definite covariance matrix, and to use iterative projection methods to solve the corresponding regularized problem.
However, a main challenge for many of these iterative methods is that the prior covariance matrix must be known and fixed (up to a constant) before starting the solution process.  
In this paper, we develop hybrid projection methods for inverse problems with mixed Gaussian priors where the prior covariance matrix is a convex combination of matrices and the mixing parameter and the regularization parameter do not need to be known in advance. 
Such scenarios may arise when data is used to generate a sample prior covariance matrix (e.g., in data assimilation) or when different priors are needed to capture different qualities of the solution.
The proposed hybrid methods are based on a mixed Golub-Kahan process, which is an extension of the generalized Golub-Kahan bidiagonalization, and a distinctive feature of the proposed approach is that \textit{both} the regularization parameter and the weighting parameter for the covariance matrix can be estimated automatically during the iterative process. Furthermore, for problems where training data are available, various data-driven covariance matrices (including those based on learned covariance kernels) can be easily incorporated.  Numerical examples from tomographic reconstruction demonstrate the potential for these methods.

\end{abstract}


\section{Introduction} \label{sec:introduction}
For many imaging systems, the ability to obtain good image reconstructions from observed data requires the inclusion of a suitable prior.
 
Priors provide a systematic and efficient means to describe in probabilistic terms any prior knowledge about the unknowns. Oftentimes prior knowledge will come from a \textit{combination} of sources, and striking a good balance of information is critical.  For example, priors may be learned from available training data, but bias in the reconstructions can be a big concern (e.g., when the training set is small or the desired image is very different from the training set).  Thus, a safer approach is to include a prior that combines learned information with conventional smoothness properties. In other scenarios (e.g. in seismic tomography), the desired solution may consist of components with different smoothness properties, and the correct mixture of smoothness priors can be difficult to know a priori. 
Using mixed Gaussian priors, where the prior covariance matrix can be represented as a convex combination of matrices, is a common approach to incorporate different prior covariance matrices.  However, various computational challenges arise for problems where the number of unknowns is very large and the regularization and mixing parameter are not known in advance.  We address these challenges by developing hybrid iterative projection methods for the efficient computation of solutions to inverse problems with mixed Gaussian priors.  By exploiting a project-then-regularize framework, we enable statistical optimization tools for selecting the regularization parameter and the mixing parameter automatically, which would be very costly for the original problem.

We are interested in linear inverse problems of the form,
\begin{equation}
\label{eq:problem}
\bfd = \bfA\bfs + \bfepsilon
\end{equation}
where $\bfd \in \mathbb{R}^{m}$ contains the observed data, $\bfA \in \mathbb{R}^{m \times n}$ models the forward process, $\bfs \in \mathbb{R}^{n}$ represents the desired parameters, and $\bfepsilon \in \mathbb{R}^{m}$ represents noise in the data. We assume that $\bfepsilon \sim \mathcal{N}(\bfzero,\bfR)$, where $\bfR$ is a symmetric positive definite matrix whose inverse and square root are inexpensive (e.g., a diagonal matrix). The goal of the inverse problem is to compute an approximation of $\bfs$, given $\bfd$ and $\bfA$.

Due to ill-posedness, small errors in the data may lead to large errors in the computed
approximation of $\bfs$, and regularization is required to stabilize the inversion process. We follow a Bayesian framework, where we assume a prior for $\bfs$. That is, we treat $\bfs$ as a Gaussian random variable with mean 
vector

$\bfmu \in \mathbb{R}^{n}$ and covariance matrix $\bfQ\in \bbR^{n \times n}$. That is,  $\bfs \sim \mathcal{N}(\bfmu,\lambda^{-2}\bfQ) $, where 
$\lambda$ is a scaling parameter (yet to be determined) for the precision matrix.

In many applications, the choice of $\bfQ$ is pre-determined (e.g., using expert knowledge) and is chosen to enforce smoothness or regularity conditions on the solution \cite{chung2017generalized,hochstenbach2010iterative,buccini2017iterated}. However, in some cases, there is not enough information to determine $\bfQ$ completely or expensive procedures are needed to determine an informative subset of covariates from a set of candidates (e.g., in geophysical imaging \cite{yadav2016statistical, yao1999calculating, zhang1995estimation}).  These scenarios motivate us to consider mixed Gaussian priors, where the covariance matrix can be represented as a convex combination of matrices. Without loss of generality we consider prior covariance matrices of the form,
\begin{equation}
\label{eq:Qsum}
\bfQ = \gamma\bfQ_1 + (1-\gamma) \bfQ_2
\end{equation}
where $\bfQ_1$ is a symmetric positive definite matrix, $\bfQ_2$ is a symmetric positive semi-definite matrix, and mixing parameter $0 < \gamma \leq 1$. We consider the case where computing matrix-vector products with $\bfQ_1$ is easy, but accessing $\bfQ_1^{-1}$ or its symmetric factorization (e.g., Cholesky or eigenvalue factorization) is not feasible.  Such scenarios arise, for example, when the prior covariance matrix is modeled entry-wise using covariance kernels.  In such cases, the main challenge is that the resulting covariance matrices are large and dense, and factorizing or inverting them can be computationally prohibitive.  However, matrix-vector multiplications can be done efficiently (e.g., via FFT embedding).  A wide range of kernels, including nonseparable spatio-temporal kernels \cite{chung2017generalized}, can be included. We assume that matrix-vector products with $\bfQ_2$ can be done efficiently. 

Covariance matrices of the form~\eqref{eq:Qsum} are becoming more common, especially in modern imaging applications where data (e.g., in the form of training images) are playing a larger role in the development of reconstruction algorithms \cite{arridge2019solving}.
Suppose we are given a dataset consisting of $N$ samples, $\bfs^{(i)} \in \mathbb{R}^{n}, i=1,2,\ldots,N$. 
Then the training data can be used to obtain an unbiased estimator of an $n \times n$ sample covariance matrix,
\begin{equation}
\label{eq:covar-est}
\widehat\bfQ = \frac{1}{N}\sum_{i=1}^{N}(\bfs^{(i)}-\bar{\bfs})(\bfs^{(i)} - \bar{\bfs})\t,
\end{equation}
where $\bar{\bfs} = \frac{1}{N}\sum_{i=1}^N \bfs^{(i)}$ is the sample mean.
Notice that 
 $\widehat\bfQ = \bfS \bfS\t$, where the symmetric factor is defined as $\bfS=\frac{1}{\sqrt{N}}\left(\begin{bmatrix}  \bfs^{(1)} & \dots & \bfs^{(N)} \end{bmatrix} -\bar{\bfs}\otimes\textbf{1}^\top\right)$ with $\textbf{1}\in\R^N$ denoting the vector whose elements are all $1$. For any vector $\bfx \in \bbR^n$, multiplication with $\widehat \bfQ$ can be done efficiently if $N << n$, e.g., using the following order of operations $\bfS(\bfS\t \bfx)$. However, notice that $\widehat\bfQ$ is likely positive semi-definite rather than positive definite, so it is common to use $\widehat \bfQ + \gamma \bfI$ where $\gamma$ is a nudging term.  Such approaches are known as sample based priors \cite{calvetti2005priorconditioners}.   Another common approach is to use a convex combination, i.e., the prior covariance matrix is given as
\begin{equation}
  \label{eq:convexcomb}
  \bfQ = \gamma \bfD + (1-\gamma) \widehat\bfQ
\end{equation}
where $\bfD$ is chosen to be the identity matrix or a suitably chosen diagonal or correlation matrix, which ensures that $\bfQ$ is positive definite, and $\gamma \in \bbR$ is called the mixing parameter.  The matrix in \eqref{eq:convexcomb} is called a shrinkage estimator of the covariance matrix \cite{schafer2005shrinkage}.
It is worth noting that covariance matrices of the form ~\eqref{eq:convexcomb} are also used in hybrid methods for data assimilation that combine an ensemble Kalman filter system with a variational (e.g., 3D-Var) system \cite{asch2016data}.  These methods require careful tuning of the so-called blending parameter $\gamma$, and many of the existing approaches require $\gamma$ to be fixed in advance. We do not assume this.

Previous works on combining training data with regularization techniques typically follow an optimal experimental design or empirical Bayes risk minimization framework \cite{haber2003learning, chung2011designing}.  More recently, there has been significant work on using training data in the context of machine learning to learn regularization functionals (e.g., \cite{li2018nett,schwab2018deep}) or to learn the ``invisible'' regions (e.g., \cite{bubba2019learning}).  The area of data-driven machine learning is currently a hot topic \cite{lucas2018using,arridge2019solving}, where the main goal is to determine new ways to combine physical models with deep learning techniques. In this work, we incorporate training data in a Bayesian framework and exploit tools from numerical linear algebra not only to compute solutions efficiently but also to determine the appropriate weighting of the training data.

In this paper we develop a hybrid iterative projection method that is based on a mixed, generalized Golub-Kahan process to approximate the MAP estimate,
\begin{equation}
	\label{eq:MAP}
	\bfs_{\rm MAP} = \argmin_\bfs \frac{1}{2}\norm[\bfR^{-1}]{\bfA\bfs - \bfd}^2 + \frac{\lambda^2}{2} \norm[\bfQ^{-1}]{\bfs - \bfmu}^2.
\end{equation}
where $\bfQ$ is of the form~\eqref{eq:Qsum}. Our approach can handle a wide range of scenarios, including data-informed regularization terms that use training or test images to define the prior. We assume that $\gamma$ is not known in advance and neither the inverse nor the factorization of $\bfQ$ is available. The proposed method has two distinctive features. First, we assume that \emph{both} $\gamma$ and $\lambda$ are unknown a priori and we estimate them during the solution process.  For problems where $\gamma$ is fixed in advance, generalized hybrid methods \cite{chung2017generalized} can be directly applied.  However, developing a hybrid method where $\gamma$ can be selected adaptively is not an obvious extension. 
We develop an iterative hybrid approach

where the problem is projected onto generalized Krylov subspaces of small but increasing dimension and the regularization parameter and mixing parameter can be simultaneously and automatically selected. Second, we describe and investigate various scenarios where training data can be used to define $\bfQ_1$ and $\bfQ_2$, so
our approach can be considered a learning approach for the regularization term.

An outline for the paper is as follows.  In Section~\ref{sec:background} we provide some background on Gaussian priors and focus on various data-driven prior covariance matrices.  Then in Section~\ref{sec:agenGK}, we describe mixed, generalized hybrid projection methods for approximating the MAP estimate~\eqref{eq:MAP}, where $\bfQ$ is of the form~\eqref{eq:Qsum}.  The approach consists of two-steps: (1) Project the problem onto a subspace of small but increasing dimension using an extension of the generalized Golub-Kahan bidiagonalization approach. (2) Solve the projected problem where the regularization parameter $\lambda$ and mixing parameter $\gamma$ can be selected automatically. Various regularization paremeter selection techniques will be investigated, and some theoretical results will be provided. 
In Section~\ref{sec:numerics} numerical results on various image processing applications show the potential benefits and flexibility of these methods.  Conclusions are provided in Section~\ref{sec:conclusions}.

\section{Mixed Gaussian priors} 
\label{sec:background}
In this section, we motivate the need for mixed Gaussian priors and draw some connections to existing works on multi-parameter Tikhonov regularization and shrinkage estimation.

To begin, we focus on using Gaussian random fields to represent prior information and summarize some common choices for the (unscaled) prior covariance matrix $\bfQ$.  Oftentimes, the covariance matrix is generated using a covariance function (also called a kernel function).  Covariance functions are crucial in many fields and 

encode assumptions about the form of the function that we are modeling.  In most cases, the prior covariance matrix $\bfQ$ is large and dense with entries directly computed as $\bfQ_{ij} =  \kappa(\bfz_{i}, \bfz_j)$, where $\{\bfz_{i}\}_{i=1}^{n}$ are the spatial points in the domain and $\kappa(\cdot,\cdot)$ is a covariance kernel function.  Some commonly used parametric covariance functions \cite{rasmussen2003gaussian} are provided in Table~\ref{cov_table}.  
\begin{table}[bthp]
\centering
\begin{tabular}{ |c|c|} 
 \hline
  & covariance kernel function  \\ 
 \hline
 squared exponential & $\text{exp}\left(-\frac{r^2}{2\ell^2}\right)$\\[5pt]
 Mat$\acute{\text{e}}$rn & $\frac{1}{2^{\nu -1}\Gamma(\nu)} \left( \frac{\sqrt{2\nu }r}{\ell}\right)^{\nu}K_{\nu}\left(\frac{\sqrt{2\nu} r}{\ell}\right)$ \\[5pt]

$\gamma-$exponential & $\text{exp}\left(-\left(\frac{r}{\ell}\right)^{\gamma}\right)$ \\[5pt]
rational quadratic & $\left(1 + \frac{r^2}{2\nu\ell^2}\right)^{-\nu}$ \\[5pt]
sinc & $\frac{\sin(\nu r)}{\nu r}$ \\[5pt]
 \hline
\end{tabular}
\caption{Summary of commonly-used covariance functions. The covariance functions are written either as functions of $\bfz_{i}$ and $\bfz_{j}$, or as a function of $r = |\bfz_{i} - \bfz_{j}|$ and depend on $\ell$ or $\ell$ and $\nu$.  $\Gamma$ is the Gamma function and $K_{\nu}(\cdot)$ is the modified Bessel function of the second kind of order $\nu$.}
\label{cov_table}
\end{table}

For some kernel choices, the precision matrix (i.e., the inverse of the covariance matrix) is sparse or structured, so working with $\bfQ^{-1}$ or its symmetric factorization has obvious computational advantages.

However, in many applications, the precision matrix is not readily available, and the aim is to develop computational methods that can work with $\bfQ$ directly and avoid the need for the inverse or symmetric factorization. Such covariance kernels may arise
in dynamic scenarios with nonseparable, spatio-temporal priors \cite{chung2018efficient, long2011state, galkaspatiotemporal} or from spatially-variant priors \cite{dong2018tomographic, yang2016spatially}.  
It is worth mentioning that in a truly Bayesian framework, the regularization parameter and the covariance kernel parameters could be included as hyperparameters and explored using MCMC methods \cite{bardsley2018computational}, but the computational costs of this approach would be very high.

One reason to use Gaussian mixtures as prior distributions is that it allows greater flexibility in the definition of the prior.  In this paper, we consider a mixture of two Gaussians, but one could consider more general mixtures.  
From a statistical viewpoint, a general formulation with $N$ Gaussian random vectors would correspond to a sum of covariance matrices.
That is, let $\bfx_1, ..., \bfx_N$ be $N$ mutually independent $n\times 1$ normal random vectors having means $\bfmu_1, ... \bfmu_N$ and covariance matrices $\bfV_1, ...\bfV_N.$ Let $\bfB_1,...\bfB_N$ be real $L \times n$ full rank matrices.  Then the $L \times 1$ random vector
\begin{equation}
    \bfy = \sum_{i=1}^N \bfB_i \bfx_i
\end{equation}
has a normal distribution with mean $\bbE \bfy = \sum_{i=1}^N \bfB_i \bfmu_i$ and covariance matrix of the form $Cov(\bfy) = \sum_{i=1}^N \bfB_i \bfV_i \bfB_i\t.$ Thus, a Gaussian mixture prior corresponds to an assumption that the desired solution can be represented as a linear combination of Gaussian realizations (e.g., with different smoothness properties).

In the context of inverse problems, we point out a connection between mixed Gaussian priors and multi-parameter Tikhonov regularization.  The basic idea of multi-parameter Tikhonov regularization, see e.g. \cite{Wang2012,LuPereverzev2011,BazanBorgesFrancisco2012,GazzolaNovati2013}, is to solve a problem of the form,
\begin{equation}
    \min_\bfs \norm[\bfR^{-1}]{\bfA \bfs - \bfd}^2
    + \sum_{i=1}^N \lambda_i^2 \norm[2]{\bfL_i \bfs}^2,
\end{equation}
 where $\lambda_i \in \bbR$ is the regularization parameter corresponding to regularization matrix $\bfL_i$ for $i=1, \ldots, N$. By including multiple penalty terms, this approach can enforce different smoothness properties (e.g, at different frequency bands) and avoid difficulties in having to select just one regularization matrix.
 In a Bayesian framework, the multi-parameter Tikhonov solution can be interpreted as a MAP estimate, under the assumption of a Gaussian prior with mean $\bfzero$ and covariance matrix 
 $ \left(\sum_{i=1}^N \lambda_i^2\bfL_i\t \bfL_i \right)^{-1}
 $. Notice that except for in very limited scenarios, this is not the same as using mixed Gaussian priors, since here the precision matrix (not the covariance matrix) is represented as a sum of matrices.

\subsection{Data-driven prior covariance matrices}
With the increasing amount of and access to data in many applications, an important and challenging task is to determine how to efficiently and effectively incorporate prior knowledge in the form of training data both in the solution computation process and the subsequent data analyses.  In this section, we describe various examples where training data can be used to define the prior covariance matrix.  For all cases, we assume that training data is provided and the sample covariance matrix \eqref{eq:covar-est} has the form $\widehat\bfQ = \bfS \bfS\t$.

As described in the introduction, the most common approach is to take $\bfQ_2 = \widehat\bfQ$ and $\bfQ_1 = \bfD$ where $\bfD$ is easy to invert (e.g., diagonal or identity matrix). In this case, a very popular approach called shrinkage estimation of covariance matrices, or more general biased estimation, can be used to reduce the variance of the estimator.  Typical shrinkage targets are diagonal matrices (e.g., including the identity matrix), and approaches to estimate the optimal shrinkage intensity $\gamma$ have been proposed by Ledoit and Wolf, Rao and Blackwell, and others \cite{LW04,asch2016data,schafer2005shrinkage,chen2009shrinkage}. 

Another approach to incorporate training data is to force some structure or functional form on the prior covariance kernel function.  For kernel functions that depend on a few parameters, the training data can be used to estimate these parameters.  A similar idea was considered in \cite{haber2003learning} where training data was used to learn parameters defining the regularization functional. However, that approach requires solving an expensive constrained optimization problem, and the learned regularization functional is tailored to the forward operator and the noise level. We consider the case where the training data come from a prior defined by a covariance kernel function (e.g., for simplicity, we consider Mat\'ern kernels). We use the training data to learn the parameters defining the prior.  This reduces to an optimization problem where the goal is to learn two parameters $\nu$ and $\ell$ from the training data by solving the optimization problem,
    \begin{equation}
        (\hat \nu, \hat \ell) = \argmin_{\nu>0,\ell>0}\norm[F]{\bfQ(\nu,\ell) - \widehat \bfQ}^2.
    \end{equation}
Once the parameters are computed, they can be used to define $\bfQ_1 = \bfQ(\hat\nu,\hat\ell)$, which can be used directly in generalized hybrid methods, or they can be combined with the sample covariance matrix, i.e., $\bfQ$ as in~\eqref{eq:Qsum} with $\bfQ_1 = \bfQ(\hat\nu,\hat\ell)$ and $\bfQ_2 = \widehat \bfQ$, and solvers described in Section \ref{sec:agenGK} can be used.

Next, we describe some computationally efficient methods to estimate $\hat \nu$ and $\hat \ell.$

Notice that
\begin{align}
    \|\bfQ(\nu,\ell) - \widehat{\bfQ}\|^2_F & = \trace(\bfQ(\nu,\ell) - \widehat{\bfQ})^\top(\bfQ(\nu,\ell) - \widehat{\bfQ}) \\
    & =  \bbE( \| (\bfQ(\nu,\ell) - \widehat{\bfQ}) \bfxi \|^2_2)
   \end{align}
 where $\bfxi$ is a random variable such that $\bbE \bfxi = \bfzero$ and $\bbE (\bfxi \bfxi\t) = \bfI$.  Although stochastic optimization methods \cite{shapiro2009lectures} could be use here, we follow an approximation approach where we use a Hutchinson trace estimator.  That is, we let  $\bfxi^{(i)} \in \bbR^n$ for $i=1,2,\ldots,M$ be realizations of a Rademacher distribution (i.e., $\bfxi$ consists of $\pm1$ with equal probability), and we consider the approximate optimization problem,
 \begin{equation}
    \label{ref:RademacherEst}
    (\check{\nu},\check{\ell}) =\argmin_{\nu>0,\ell>0} 
    \frac{1}{M}\sum_{i=1}^{M}  \|(\bfQ(\nu,\ell) - \widehat\bfQ)\bfxi^{(i)} \|^2_2.
\end{equation}

We mention that for problems without training data, semivariogram hyperparameters were investigated in \cite{bardsley2018semivariogram} to estimate Mat{\' e}rn parameters from the data.

\section{Hybrid projection methods for mixed Gaussian priors}
\label{sec:agenGK}
In this section, we describe a hybrid projection method to approximate the MAP estimate~\eqref{eq:MAP}. The distinguishing factor of this approach compared to generalized Golub-Kahan (genGK) hybrid methods \cite{chung2017generalized} is that we address problems where the prior covariance matrix is of the form~\eqref{eq:Qsum}. That is, we consider priors of the form $\bfs \sim\calN(\bfmu, \lambda^{-2}(\gamma \bfQ_1 + (1-\gamma) \bfQ_2))$, and exploit a hybrid projection framework to enable tools for selecting both the regularization parameter $\lambda$ and the mixing parameter $\gamma$ simultaneously.
Using the following change of variables,
$$\bfx = \bfQ^{-1}(\bfs -\bfmu), \quad \bfb = \bfd - \bfA \bfmu,$$
we see that solving~\eqref{eq:MAP} is equivalent to solving
\begin{equation}
	\label{eq:transformed}
	\min_\bfx \frac{1}{2}\norm[\bfR^{-1}]{\bfA\bfQ\bfx - \bfb}^2 + \frac{\lambda^2}{2} \norm[\bfQ]{\bfx}^2.
\end{equation}

If $\gamma$ is known in advance, we can directly apply the genGK hybrid method and estimate $\lambda$ automatically \cite{chung2017generalized}. However, in many cases, we don't know $\gamma$ in advance, so we want to estimate $\gamma$ during the iterative process.  For this, we develop a variant of the genGK bidiagonlization which we call a \emph{mixed} Golub-Kahan (mixGK) process.
Each iteration of the mixGK process requires two steps.  The first step is to run one iteration of the genGK bidiagonalization process with $\bfQ_1$. The second step incorporates $\bfQ_2$ so that the regularized problem can be iteratively projected onto a smaller subspace, and $\gamma$ and $\lambda$ can \textit{both} be selected automatically.  Next we describe the mixGK process in detail.

Given matrices $\bfA$, $\bfR$, $\bfQ_1$, and vector $\bfb,$ with initializations $\beta_1 = \norm[\bfR^{-1}]{\bfb}$, $\bfu_1 = \bfb/\beta_1$ and $\alpha_1 \bfv_1 = \bfA\t \bfR^{-1} \bfu_1$,
 the $k$th iteration of the genGK bidiagonalization procedure with $\bfQ_1$ generates vectors $\bfu_{k+1}$ and $\bfv_{k+1}$ such that
\begin{align*}
	\beta_{k+1} \bfu_{k+1} & = \bfA \bfQ_1 \bfv_k -\alpha_k \bfu_k\\
		\alpha_{k+1} \bfv_{k+1} & = \bfA\t \bfR^{-1} \bfu_{k+1} -\beta_{k+1} \bfv_k,
	\end{align*}
where scalars $\alpha_i, \beta_i \geq 0$ are chosen such that $\norm[\bfR^{-1}]{\bfu_i} = \norm[\bfQ_1]{\bfv_i} = 1$. At the end of $k$ steps, we have
\[ \bfB_k \equiv \> \begin{bmatrix}
\alpha_1 \\ \beta_2 & \alpha_2 \\ & \beta_3 & \ddots \\ & & \ddots & \alpha_k \\ & & & \beta_{k+1}
\end{bmatrix}\,,  \qquad \bfU_{k+1} \equiv [\bfu_1,\dots,\bfu_{k+1}],\quad \mbox{and} \quad \bfV_k \equiv [\bfv_1,\dots,\bfv_k],\]
where the following relations hold up to machine precision,
\begin{align}\label{e_bk}
\bfU_{k+1}\beta_1 \bfe_1  =  &\> \bfb \\ \label{e_vk}
\bfA \bfQ_1 \bfV_k = & \>\bfU_{k+1} \bfB_k \\ \label{e_uk}
\bfA\t \bfR^{-1} \bfU_{k+1} = & \> \bfV_k \bfB_k\t + \alpha_{k+1}\bfv_{k+1}\bfe_{k+1}\t\,.
\end{align}
Furthermore, in exact arithmetic, matrices $\bfU_{k+1}$ and $\bfV_k$ satisfy the following orthogonality conditions
\begin{equation}
	\label{eq:orthog}
	\bfU_{k+1}\t \bfR^{-1} \bfU_{k+1} = \bfI_{k+1} \qquad \mbox{and} \qquad \bfV_k\t \bfQ_1 \bfV_k = \bfI_k.
\end{equation}
If we let $\widetilde{\bfU}_{k+1} = \bfL_{\bfR}\bfU_{k+1}$ where $\bfR^{-1} = \bfL\t_{\bfR}\bfL_{\bfR}$, then $\widetilde{\bfU}\t_{k+1}\widetilde{\bfU}_{k+1} = \bfI_{k+1}$.

Next, in order to incorporate $\bfQ_2$, we additionally compute $m \times k$ matrix $\bfL_{\bfR}\bfA \bfQ_2 \bfV_k$. Assuming that the columns of $\widetilde{\bfU}_{k+1}$ and $\bfL_{\bfR}\bfA \bfQ_2 \bfV_k$ are linearly independent, we can compute the skinny QR factorization, $(\bfI - \widetilde{\bfU}_{k+1} \widetilde{\bfU}_{k+1}\t)\bfL_{\bfR}\bfA \bfQ_2 \bfV_k = \bfY_k \bfR_k$ where $\bfY_{k}\in\bbR^{m \times k}$
contains orthonormal columns and $\bfR_k \in \bbR^{k\times k }$ is upper triangular.  Notice that since column vectors in $\bfY_k$ and $\widetilde \bfU_{k+1}$ are orthogonal, we get the skinny QR factorization,
\begin{equation}
	\label{eq:skinnyQR}
	\begin{bmatrix}
		\widetilde{\bfU}_{k+1} & \bfL_{\bfR}\bfA \bfQ_2 \bfV_k
	\end{bmatrix} = \begin{bmatrix}
		\widetilde{\bfU}_{k+1} & \bfY_k
\end{bmatrix}
\begin{bmatrix}
	\bfI_{k+1} & \widetilde{\bfU}_{k+1}\t \bfL_{\bfR}\bfA \bfQ_2 \bfV_k \\
	\bfzero & \bfR_k
\end{bmatrix}.
\end{equation}
The mixGK process is summarized in Algorithm~\ref{alg:agenGK}.

\begin{algorithm}[bthp]
\begin{algorithmic}[1]
\REQUIRE Matrices $\bfA$, $\bfR$, $\bfQ_1$ and $\bfQ_2$, and vector $\bfb$.
\STATE $\beta_1 \bfu_1 = \bfb,$ where $\beta_1 = \norm[\bfR^{-1}]{\bfb}$
\STATE $\alpha_1 \bfv_1 = \bfA\t \bfR^{-1}\bfu_1$
\FOR {$k=1, 2, \dots$}
\STATE $\beta_{k+1}\bfu_{k+1} = \bfA{\bfQ_1}\bfv_k - \alpha_k \bfu_k$, where $\beta_{k+1} = \norm[\bfR^{-1}]{\bfA{\bfQ_1}\bfv_k - \alpha_k \bfu_k}$
\STATE $\alpha_{k+1}\bfv_{k+1} = \bfA\t \bfR^{-1} \bfu_{k+1} - \beta_{k+1} \bfv_k$, where $\alpha_{k+1} = \norm[\bfQ_1]{\bfA\t \bfR^{-1} \bfu_{k+1} - \beta_{k+1} \bfv_k}$
\STATE $[\bfY_k, \bfR_k] = qr((\bfI - \widetilde{\bfU}_{k+1}\widetilde{\bfU}_{k+1}\t)\bfL_{R}\bfA \bfQ_2 \bfV_k, 0);$
\ENDFOR
\end{algorithmic}
\caption{mixed Golub-Kahan (mixGK) process}
\label{alg:agenGK}
\end{algorithm}
Notice that in addition to the computational cost of the genGK bidiagonalization, which includes one matrix-vector product with $\bfA$, one with $\bfA\t$, two with $\bfQ_1$, and two solves with $\bfR$, each iteration of the mixGK process requires one matrix-vector product with $\bfQ_2$ and a QR factorization in step 6.  Instead of performing a standard QR factorization on an $m$-by-$k$ matrix, an efficient rank-one update strategy can be used to alleviate the computational cost. More specifically, we will describe it using mathematical induction. Let
\begin{equation}
\label{eq:rankone}
(\bfI - \widetilde{\bfU}_{k}\widetilde{\bfU}_{k}\t)\bfL_{R}\bfA \bfQ_2 \bfV_{k-1} = \bfY_{k-1}\bfR_{k-1}
\end{equation}
be the skinny QR factorization, 
where $\bfY_{k-1}\t \bfY_{k-1} = \bfI_{k-1}$ and $\bfR_{k-1}$ is an upper triangular matrix.
Define $\widetilde{\bfU}_{k+1} = \begin{bmatrix}
    \widetilde{\bfU}_{k} & \widetilde{\bfu}_{k+1}
\end{bmatrix}$  and $\bfV_{k} = \begin{bmatrix}
    \bfV_{k-1} & \bfv_{k}
\end{bmatrix}$. Then by (\ref{eq:rankone}), we have
\begin{align*}
(\bfI - \widetilde{\bfU}_{k+1}\widetilde{\bfU}_{k+1}\t)\bfL_{R}\bfA \bfQ_2 \bfV_{k} &= \begin{bmatrix}
    (\bfI - \widetilde{\bfU}_{k+1}\widetilde{\bfU}_{k+1}\t)\bfL_{R}\bfA \bfQ_2 \bfV_{k-1} & (\bfI - \widetilde{\bfU}_{k+1}\widetilde{\bfU}_{k+1}\t)\bfL_{R}\bfA \bfQ_2 \bfv_{k}
\end{bmatrix}\\
&=\begin{bmatrix}
    (\bfI - \widetilde{\bfU}_{k}\widetilde{\bfU}_{k}\t - \widetilde{\bfu}_{k+1}\widetilde{\bfu}_{k+1}\t )\bfL_{R}\bfA \bfQ_2 \bfV_{k-1} & (\bfI - \widetilde{\bfU}_{k+1}\widetilde{\bfU}_{k+1}\t)\bfL_{R}\bfA \bfQ_2 \bfv_{k}
    \end{bmatrix}\\
&=\begin{bmatrix}
    \bfY_{k-1}\bfR_{k-1} - \widetilde{\bfu}_{k+1}\widetilde{\bfu}_{k+1}\t \bfY_{k-1}\bfR_{k-1} & (\bfI - \widetilde{\bfU}_{k+1}\widetilde{\bfU}_{k+1}\t)\bfL_{R}\bfA \bfQ_2 \bfv_{k}
    \end{bmatrix}.
\end{align*}
Since the first matrix is a rank-one update of a QR factorization, its QR factorization can be obtained in
$\calO(mk)$ operations \cite{daniel1976reorthogonalization}.  That is, we have
$$\bfY_{k-1}\bfR_{k-1} - \widetilde{\bfu}_{k+1} (\bfR_{k-1}\t\bfY_{k-1}\t\widetilde{\bfu}_{k+1})\t =  \widehat{\bfY}_{k-1}\widehat{\bfR}_{k-1}$$
where $\widehat{\bfY}_{k-1}\t\widehat{\bfY}_{k-1} = \bfI_{k-1}$ and $\widehat{\bfR}_{k-1}$ is an upper triangular matrix.
Finally, let $\widehat{\bfv}_{k} = (\bfI - \widetilde{\bfU}_{k+1}\widetilde{\bfU}_{k+1}\t)\bfL_{R}\bfA \bfQ_2 \bfv_{k}$, then one step of the Gram-Schmidt process gives the desired QR factorization, $$\begin{bmatrix}
    \widehat{\bfY}_{k-1}\widehat{\bfR}_{k-1} & \widehat{\bfv}_{k}
\end{bmatrix} = {\bfY}_{k}{\bfR}_{k}.$$

\subsection{Solving the projected problem}
Using the mixGK process described above, we now describe a hybrid iterative projection method to solve~\eqref{eq:transformed}. In particular, we consider the projected problem,
\begin{equation}
\label{prob:projected}
	\min_{\bfx \in \calR(\bfV_k)}\frac{1}{2} \norm[\bfR^{-1}]{\bfA\bfQ\bfx - \bfb}^2 + \frac{\lambda^2}{2} \norm[\bfQ]{\bfx}^2
\end{equation}
where $\calR(\cdot)$ denotes the column space.  Let $\bfx = \bfV_k \bfy$ where $\bfy\in\bbR^k$. Then using the relationships from the mixGK process, we obtain the equivalent problems,
\begin{align}
\min_\bfy & \frac{1}{2} \norm[\bfR^{-1}]{\gamma\bfA \bfQ_1 \bfV_k \bfy + (1-\gamma) \bfA \bfQ_2 \bfV_k \bfy - \bfb}^2 + \frac{\lambda^2}{2} \bfy\t \bfV_k\t (\gamma\bfQ_1+ (1-\gamma) \bfQ_2) \bfV_k \bfy\\
	\min_\bfy & \frac{1}{2}\norm[2]{\gamma\widetilde{\bfU}_{k+1} \bfB_k \bfy + (1-\gamma) \bfL_{\bfR}\bfA \bfQ_2 \bfV_k \bfy - \bfL_{\bfR}\bfb}^2 + \frac{\lambda^2\gamma}{2} \bfy\t \bfy + \frac{\lambda^2 (1-\gamma)}{2} \bfy\t \bfV_k\t \bfQ_2 \bfV_k \bfy \\
	\min_\bfy & \frac{1}{2}\norm[2]{\begin{bmatrix}
	\widetilde{\bfU}_{k+1} & \bfL_{\bfR}\bfA \bfQ_2 \bfV_k \end{bmatrix} \begin{bmatrix}\gamma\bfB_k \\ (1-\gamma)  \bfI_k \end{bmatrix}\bfy - \bfL_{\bfR}\bfb}^2 + \frac{\lambda^2\gamma}{2} \norm[2]{\bfy}^2 + \frac{\lambda^2 (1-\gamma)}{2} \bfy\t \bfV_k\t \bfQ_2 \bfV_k \bfy.
	\label{eq:notorthog}
\end{align}
Using equation (\ref{eq:skinnyQR}) and
the fact that
\begin{equation}
	\begin{bmatrix}
		\widetilde{\bfU}_{k+1} & \bfY_{k}
	\end{bmatrix} \begin{bmatrix}
	\beta_1 \bfe_1 \\ \bf0
	\end{bmatrix} = \widetilde{\bfU}_{k+1} (\beta_1 \bfe_1) = \bfL_{\bfR}\bfb
\end{equation}
where 
$\begin{bmatrix}\widetilde{\bfU}_{k+1} & \bfY_k \end{bmatrix}$ contains orthonormal columns (so it can be taken out of the norm), the projected, regularized problem becomes
\begin{equation}
	\label{eq:projectedproblem}
	\min_\bfy \frac{1}{2}\norm[2]{\begin{bmatrix}
 \bfI_{k+1} & \widetilde{\bfU}_{k+1}\t \bfL_{\bfR}\bfA \bfQ_2 \bfV_k \\
 \bfzero & \bfR_k
 \end{bmatrix}\begin{bmatrix}\gamma\bfB_k \\ (1-\gamma)  \bfI_k \end{bmatrix}\bfy - \begin{bmatrix}
 \beta_1 \bfe_1 \\ \bf0
 \end{bmatrix}}^2 + \frac{\lambda^2\gamma}{2} \norm[2]{\bfy}^2 + \frac{\lambda^2 (1-\gamma)}{2} \bfy\t \bfV_k\t \bfQ_2 \bfV_k \bfy.
\end{equation}
Note that the solution subspace for $\bfx$ does not depend on $\gamma$ and $\lambda$, but the solution of the projection problem depends on both $\gamma$ and $\lambda$.  Let $\bfy_k(\lambda,\gamma)$ denote the solution to~\eqref{eq:projectedproblem}, then the $k$ iterate of the mixGK method is given as
\begin{equation}
	\label{eq:iterates}
	\bfs_k (\lambda, \gamma) = \bfmu + (\gamma\bfQ_1 + (1-\gamma) \bfQ_2) \bfV_k \bfy_k (\lambda,\gamma).
\end{equation}

In Section~\ref{sub:param} we describe some techniques for selecting $\lambda$ and $\gamma$ at each iteration, but first we provide a theoretical result.  We show that for fixed regularization parameter $\lambda$ and fixed mixing parameter $\gamma$, the proposed mixGK method converges in exact arithmetic to the desired regularized solution.
\begin{theorem}
\label{thm:convergence}
Assume $\lambda> 0$ and $0<\gamma\leq 1$. Let $\bfy_k(\lambda, \gamma)$ be the exact solution to projected problem (\ref{eq:projectedproblem}). Then the kth iterate of the mixGK approach, written as 
\begin{equation}
    \bfs_k = {\bfmu} + \bfQ\bfV_k\bfy_k(\lambda, \gamma)
\end{equation}
 converges to the MAP estimate given by
 \begin{equation}
     \bfs_{\rm MAP} = {\bfmu} + \bfQ(\bfA\t\bfR^{-1}\bfA\bfQ + \lambda^2\bfI_n)^{-1}\bfA\t\bfR^{-1}\bfb.
 \end{equation}
 \end{theorem}
 \begin{proof}
    The proof is provided in Appendix \ref{sec:appendix}.
 \end{proof}

\subsection{Regularization parameter selection methods}
\label{sub:param}
In this section, we describe two extensions of existing regularization parameter selection methods that can be used for selecting $\gamma$ and $\lambda$ at each iteration of the mixGK hybrid method. Notice that the solution at the $k$-th iteration can be written as
\begin{equation}
	\label{eq:regsoln}
	\bfs_k (\lambda, \gamma) = \bfmu + (\gamma\bfQ_1 + (1-\gamma) \bfQ_2) \bfV_k \bfy_k (\lambda,\gamma),
\end{equation}
where 
\begin{equation}
\label{eq:proj-solution}
    \begin{array}{rcl}
        \bfy_k (\lambda,\gamma) & = & \left(\bfD_k(\gamma)\t \bfD_k(\gamma) + \lambda^2 \gamma \bfI_k + \lambda^2(1-\gamma)\bfV_k\t\bfQ_2\bfV_k \right)^{-1} \bfD_k(\gamma)\t \begin{bmatrix} \beta_1\bfe_1 \\ \textbf{0} \end{bmatrix} \\
        & =& \bfC_k(\gamma, \lambda) \begin{bmatrix} \beta_1\bfe_1 \\ \textbf{0} \end{bmatrix}
    \end{array}
\end{equation}  with 
\begin{align}
\label{eq:D}
    \bfD_k(\gamma) & = \begin{bmatrix}
	    \bfI_{k+1}& \widetilde{\bfU}_{k+1}\t\bfL_{\bfR}\bfA \bfQ_2 \bfV_k \\ \textbf{0} & \bfR_k \end{bmatrix} \begin{bmatrix}\gamma\bfB_k \\ (1-\gamma)  \bfI_k \end{bmatrix}  =  \begin{bmatrix}
	        \gamma\bfB_{k} + (1-\gamma)\widetilde{\bfU}_{k+1}\t\bfL_{\bfR}\bfA \bfQ_2 \bfV_k \\
	        (1-\gamma) \bfR_{k}
	    \end{bmatrix}\\
	    \bfC_k(\gamma, \lambda) & =  \left(\bfD_k(\gamma)\t \bfD_k(\gamma) + \lambda^2\gamma\bfI_k + \lambda^2(1-\gamma)\bfV_k\t\bfQ_2\bfV_k \right)^{-1} \bfD_k(\gamma)\t.
	\end{align}
As with regularization parameter selection methods for standard hybrid methods, there is not one method that will work for all problems, so it is advised to try various approaches in practice.

In order to provide a comparison, we provide ``optimal'' parameters which are computed as
\begin{equation}\label{eq:reg_opt}
	(\gamma_{\rm opt}, \lambda_{\rm opt}) = \argmin_{0< \gamma \le 1,\,\lambda} \norm[2]{\bfs_k(\gamma, \lambda) - \bfs_{\rm true}}^2,
\end{equation}
where $\bfs_{\rm true}$ is the true solution (that is not available in practice).

\paragraph{Unbiased predictive risk estimation (UPRE).} We can select parameters $\gamma, \lambda$ such that \begin{equation}
    (\gamma_{\rm u}^{\rm proj},\lambda_{\rm u}^{\rm proj}) = \argmin\limits_{0 <\gamma\le 1,\,\lambda} \mathcal{U}_{\rm  proj}(\gamma,\lambda) = \dfrac{1}{2k+1}\|\bfr_k^{\rm  proj}(\gamma,\lambda)\|^2_2 + \dfrac{2\sigma^2}{2k+1 } {\rm tr}(\bfD_k(\gamma)\bfC_k(\gamma,\lambda)) -\sigma^2
    \label{eq:projUpre}
\end{equation}  where $\sigma^2$ is noise level, and \begin{equation}
    \label{eq:proj_res}
    \bfr_k^{\rm proj}(\gamma,\lambda) =  \bfD_k(\gamma) \bfy_k(\gamma,\lambda) -  \begin{bmatrix} \beta_1\bfe_1 \\ \textbf{0}\end{bmatrix}
\end{equation} 
 and 
\begin{equation}
\label{eq:trace_res}
\begin{array}{rcl}
    {\rm tr}(\bfD_k(\gamma)\bfC_k(\gamma,\lambda)) & = & {\rm tr}(\bfC_k(\gamma,\lambda)\bfD(\gamma)) \\
    & = & {\rm tr}(\left((\bfD_k(\gamma))\t \bfD_k(\gamma) + \lambda^2 \gamma \bfI_k + \lambda^2(1-\gamma)\bfV_k\t\bfQ_2\bfV_k \right)^{-1} (\bfD_k(\gamma))\t \bfD_k(\gamma)). \\
    \end{array}
\end{equation} When the noise level $\sigma^2$ is not provided, a noise level estimation algorithm (e.g., based on a wavelet decomposition of the observation) can be utilized \cite{donoho1995denoising}. 

\paragraph{Generalized cross validation (GCV).} Without a priori knowledge of the noise level, another option is to use an extension of the GCV method \cite{golub1979generalized,hansen2010discrete}. The basic idea is to select parameters, \begin{equation}\label{eq:proj_gcv}
	(\gamma_{\rm g}^{\rm proj}, \lambda_{\rm g}^{\rm proj}) = \argmin\limits_{0 < \gamma \le 1,\,\lambda} {\cal G}_{\rm proj}(\gamma,\lambda) = \dfrac{\|\bfr_k^{\rm proj}(\gamma,\lambda)\|_2^2}{({\rm tr}(\bfI_{2k+1} - \bfD_k(\gamma)\bfC_k(\gamma,\lambda)))^2}
\end{equation} where $\bfr_k^{\rm proj}(\gamma,\lambda)$, $\bfD_k(\gamma)$, and $\bfC_k(\gamma,\lambda)$ are same as \eqref{eq:projUpre}. \\

Notice that $\bfr_k^{\rm proj}$ and ${\rm tr}(\bfD_k(\gamma)\bfC_k(\gamma,\lambda))$ are functions of $k$ in both the GCV and UPRE functions. 
In order to prove convergence of the parameters chosen by UPRE and GCV, we begin with a lemma that shows convergence of the projected residual $\bfr_k^{\rm proj}$ and trace term ${\rm tr}(\bfD_k(\gamma)\bfC_k(\gamma,\lambda))$ to their full counterparts.
\begin{lemma}
\label{lemma:res_proj}
With \eqref{eq:proj_res}, \eqref{eq:trace_res}, if $k \rightarrow n$, then \begin{equation}
    \begin{array}{rcl}
    \bfr_k^{\rm proj}& \rightarrow & \bfr^{\rm full}(\gamma,\lambda)\\
    {\rm tr}(\bfD_k(\gamma)\bfC_k(\gamma,\lambda)) & \rightarrow & {\rm tr}(A(\gamma,\lambda))
    \end{array}
\end{equation} where \begin{equation}
    \begin{array}{rcl}
    \bfr^{\rm full}(\gamma,\lambda) & =& \bfL_{\bfR}\bfA\bfQ\bfx(\gamma,\lambda)-\bfL_{\bfR}\bfb \\
    A(\gamma,\lambda) & = & \bfL_{\bfR}\bfA\bfQ(\bfQ\t\bfA\t\bfR^{-1}\bfA\bfQ + \lambda^2\bfQ)^{-1}\bfQ\t\bfA\t\bfL_{\bfR}\t.
    \end{array}
\end{equation} and $\bfr^{\rm full}(\gamma,\lambda)=\bfr_n^{\rm proj}(\gamma,\lambda)$.
\end{lemma}
\begin{proof}
 The proof is provided in Appendix \ref{sec:appendix2}.
\end{proof}

Next we provide convergence results for the UPRE and GCV selected parameters that are similar to results provided in \cite{renaut2017hybrid} but are extended to the mixed hybrid methods. In particular, we show in Theorem \ref{thm:Convergence_U_G} that the UPRE parameters
for the projected problem converge to the UPRE parameters for the full problem.  Then, we show that with an additional weighting parameter, the same result holds for GCV parameters.

\begin{theorem}
    \label{thm:Convergence_U_G}
    From \eqref{eq:transformed}, the UPRE for the full problem is given \begin{equation}
    (\gamma_{\rm u}^{\rm full},\lambda_{\rm u}^{\rm full}) = \argmin\limits_{0 < \gamma \le 1,\,\lambda}\mathcal{U}_{\rm full}(\gamma,\lambda) = \dfrac{1}{m}\|\bfr^{\rm full}(\gamma,\lambda)\|^2_2+\dfrac{2\sigma^2}{m}{\rm tr}(A(\gamma,\lambda))-\sigma^2.
    \label{eq:fullUpre}
\end{equation}    Then, \begin{equation}
    (\gamma_{\rm u}^{\rm proj}, \lambda_{\rm u}^{\rm proj}) \rightarrow (\gamma_{\rm u}^{\rm full}, \lambda_{\rm u}^{\rm full})
\end{equation} as $k\rightarrow n$.
\end{theorem}
\begin{proof}
    Since $\|\bfr^{\rm proj}_k\|_2^2 \rightarrow \|\bfr^{\rm full}\|^2_2$ and ${\rm tr}(\bfD_k(\gamma)\bfC_k(\gamma,\lambda)) \rightarrow {\rm tr}(A(\gamma,\lambda))$ as shown in Lemma \ref{lemma:res_proj}, \begin{equation*}
        \argmin_{0 < \gamma \le 1,\,\lambda}{\cal U}_{\rm proj}(\gamma,\lambda) \rightarrow \argmin_{0 < \gamma \le 1,\,\lambda}{\cal U}_{\rm full}(\gamma,\lambda)
    \end{equation*} as $k\rightarrow n$ for the same noise level $\sigma^2$.
\end{proof} 

For the full problem, the GCV parameters are given by \begin{equation}
    (\gamma_{\rm g}^{\rm full},\lambda_{\rm g}^{\rm full}) = \argmin\limits_{0 < \gamma \le 1,\,\lambda} {\cal G}_{\rm full}(\gamma,\lambda) = \dfrac{\|\bfr^{\rm full}(\gamma,\lambda)\|_2^2}{({\rm tr}(\bfI_m - A(\gamma,\lambda)))^2}.
    \label{eq:fullGCV}
\end{equation} In contrast with UPRE, $(\gamma_{\rm g}^{\rm proj}, \lambda_{\rm g}^{\rm proj})$ does not minimize \eqref{eq:fullGCV} as $k\rightarrow n$ because the trace of $\bfI_{2k+1} -\bfD_k(\gamma)\bfC_k(\gamma,\lambda)$ does not converge to the trace of $\bfI_m-A(\gamma,\lambda)$. To compensate for this, we include an additional parameter $\omega$ in \eqref{eq:proj_gcv} as, \begin{equation}
    \label{eq:proj_wgcv}
    (\lambda^{\rm proj}_{\rm w},\gamma^{\rm proj}_{\rm w})=\argmin_{0 < \gamma \le 1,\,\lambda}{\cal W}(\gamma,\lambda)_{\rm proj}=\frac{\|\bfr_k^{\rm proj}(\gamma,\lambda)\|^2_2}{({\rm tr} (\bfI_{2k+1}-\omega\bfD_k(\gamma)\bfC_k(\gamma,\lambda)))^2}    
\end{equation} where $\omega = \frac{2k+1}{m}$. Since \begin{equation}
    ({\rm tr} (\bfI_{2k+1}-\omega\bfD_k(\gamma)\bfC_k(\gamma,\lambda)))^2 = \frac{2k+1}{m}({\rm tr} (\bfI_{m}-\bfD_k(\gamma)\bfC_k(\gamma,\lambda)))^2,
\end{equation} ${\cal G}_{\rm full}(\gamma,\lambda)$ is minimized by $(\lambda^{\rm proj}_{\rm w},\gamma^{\rm proj}_{\rm w})$ as $k\rightarrow n$. Similar modified GCV functions were considered  in \cite{chung2008weighted, renaut2017hybrid}.

\section{Numerical results} \label{sec:numerics}
 In this section, we provide various numerical results from tomography to investigate our proposed hybrid method based on the mixGK process, which we denote as `mixHyBR'. First, in Section \ref{sec:numericEx1} we investigate data-driven mixed Gaussian priors where we assume that training data are available, and we compare various hybrid methods to existing shrinkage algorithms. Then, we consider a seismic crosswell tomography reconstruction problem in Section \ref{sec:numericEx2}, where we show that using a combination of covariance kernels can result in improved reconstructions. For the stopping criteria for mixHyBR, we use a combination of approaches described in \cite{chung2008weighted, chung2015hybrid, chung2017generalized}, where the iterative process is terminated if either of the following three criteria is satisfied: (i) a maximum number of iterations is reached, (ii) depending on the chosen regularization parameter selection method,
the function (\ref{eq:projUpre}) for UPRE, 
(\ref{eq:proj_gcv}) for GCV, or (\ref{eq:proj_wgcv}) for WGCV attains a minimum or flattens out,
 and (iii) tolerances on residuals are achieved.

 \subsection{Spherical tomography example} \label{sec:numericEx1}
 For our first example, we use a spherical means tomography reconstruction problem from the IRTools toolbox \cite{IRtools,hansen2018air}.  Such models are often used in imaging problems from photoacoustic or optoacoustic imaging, which is a non-ionizing biomedical imaging modality.
 The true image $\bfs_{\rm true}$ consists of $128 \times 128$ pixels, and the forward model matrix $\bfA$ represents a ray-tracing operation along semi-circle curves where the angle of centers range from $0^\circ$ to $90^\circ$ at steps of $(90/64)^\circ$. The number of circles at each angle is $90$. Thus the dimension of $\bfA$ is $5,760\times16,384$ and the sinogram is $90 \times 64$. The simulated observed sinogram was obtained as in (\ref{eq:problem}), where we have included $3\%$ additive Gaussian white noise, i.e., $\frac{\norm{\bfepsilon}}{\norm{\bfA\bfs_{\rm true}}} = 0.03$. Other conditions are chosen as the default settings provided by the toolbox; see \cite{IRtools} for details. In the left panel of Figure \ref{fig:SphericalForwardModel}, we provide the true image along with some of the integration curves.  
 
 \begin{figure}[b!]
   \begin{center}
 \begin{tabular}{ccc}
   \raisebox{-.5\height}{\includegraphics[width=.5\textwidth]{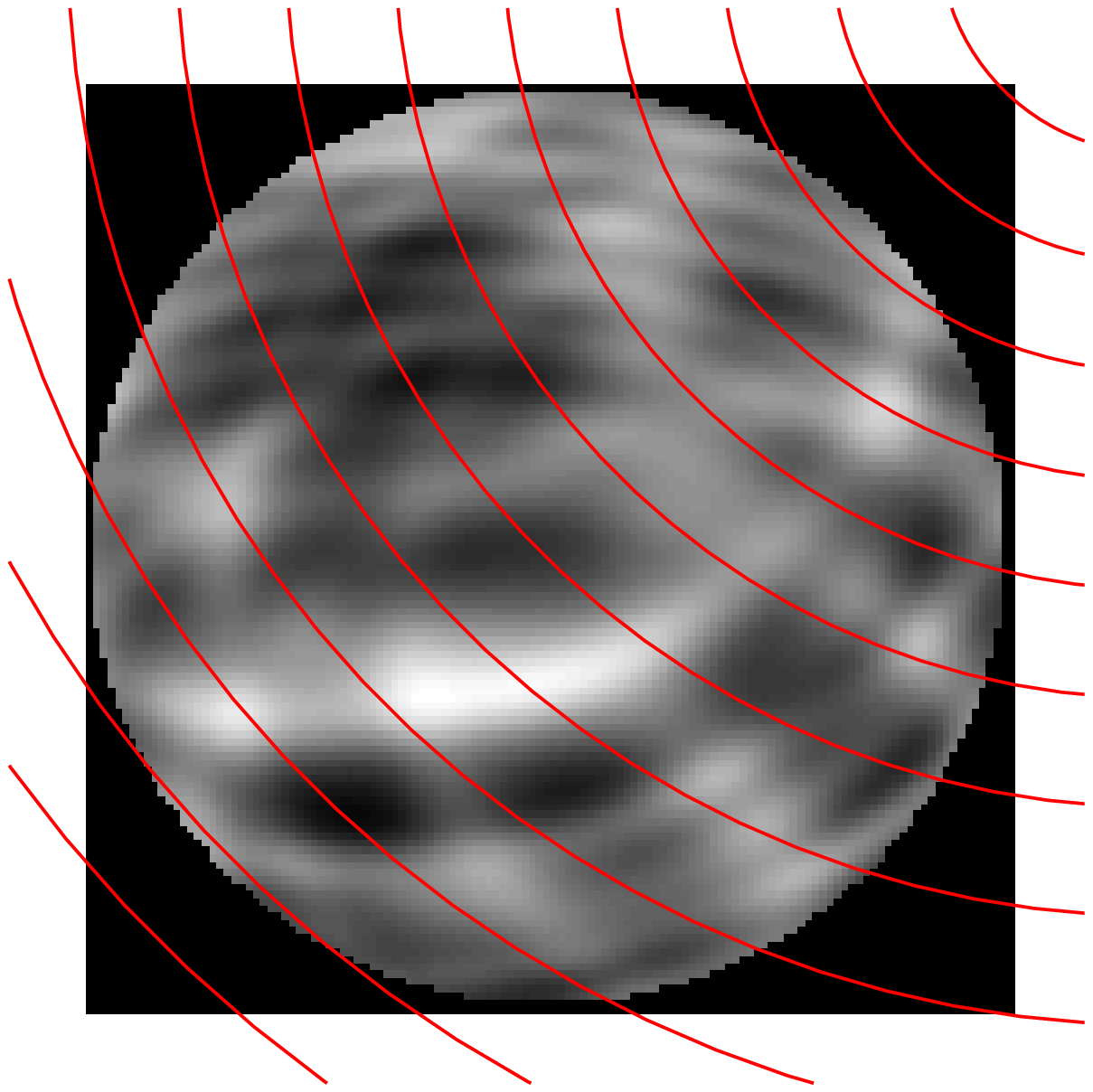}}&
   \hspace{-0.16in}\includegraphics[width=.2\textwidth]{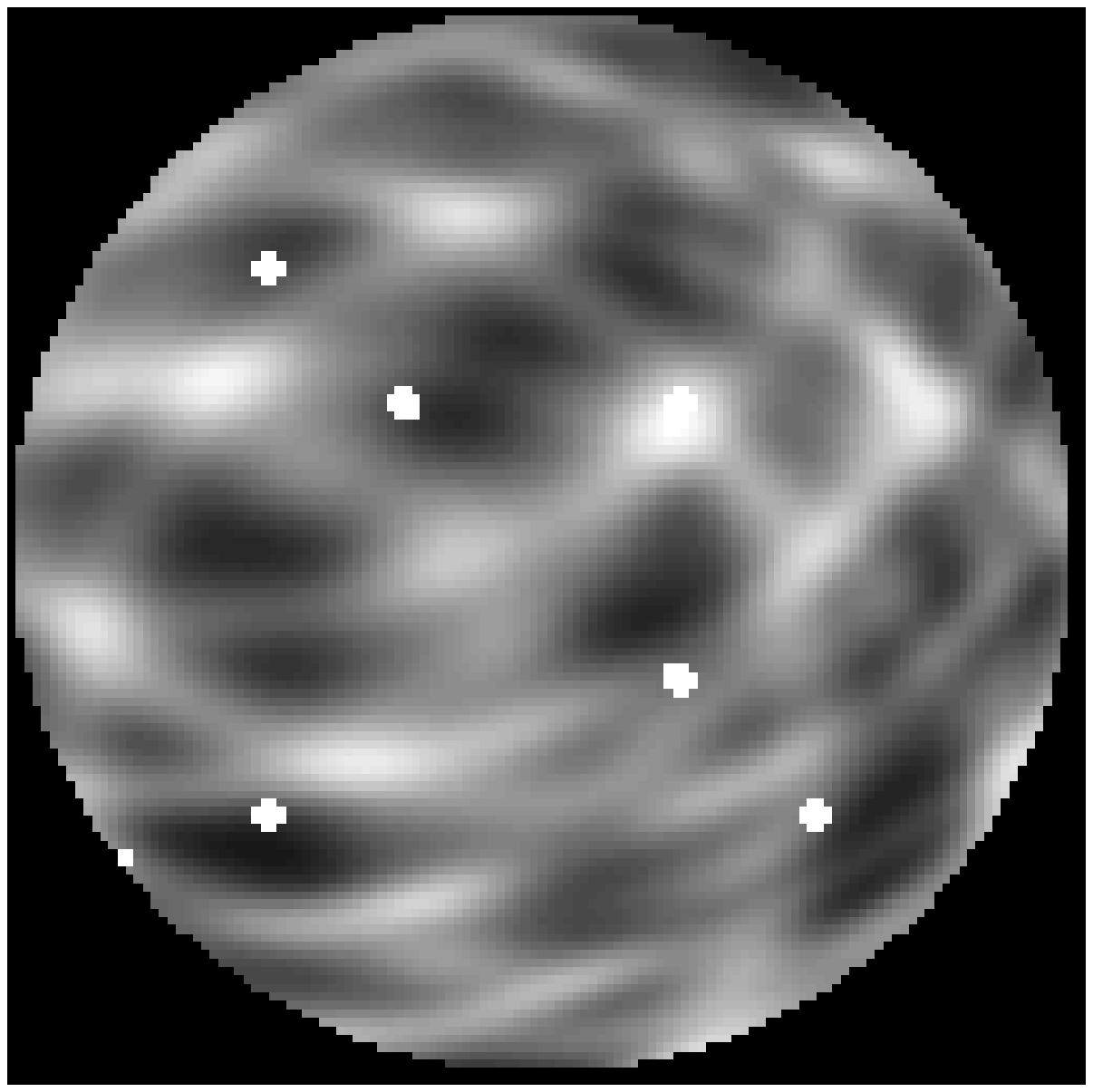}
   & \hspace{0.27in}\includegraphics[width=.2\textwidth]{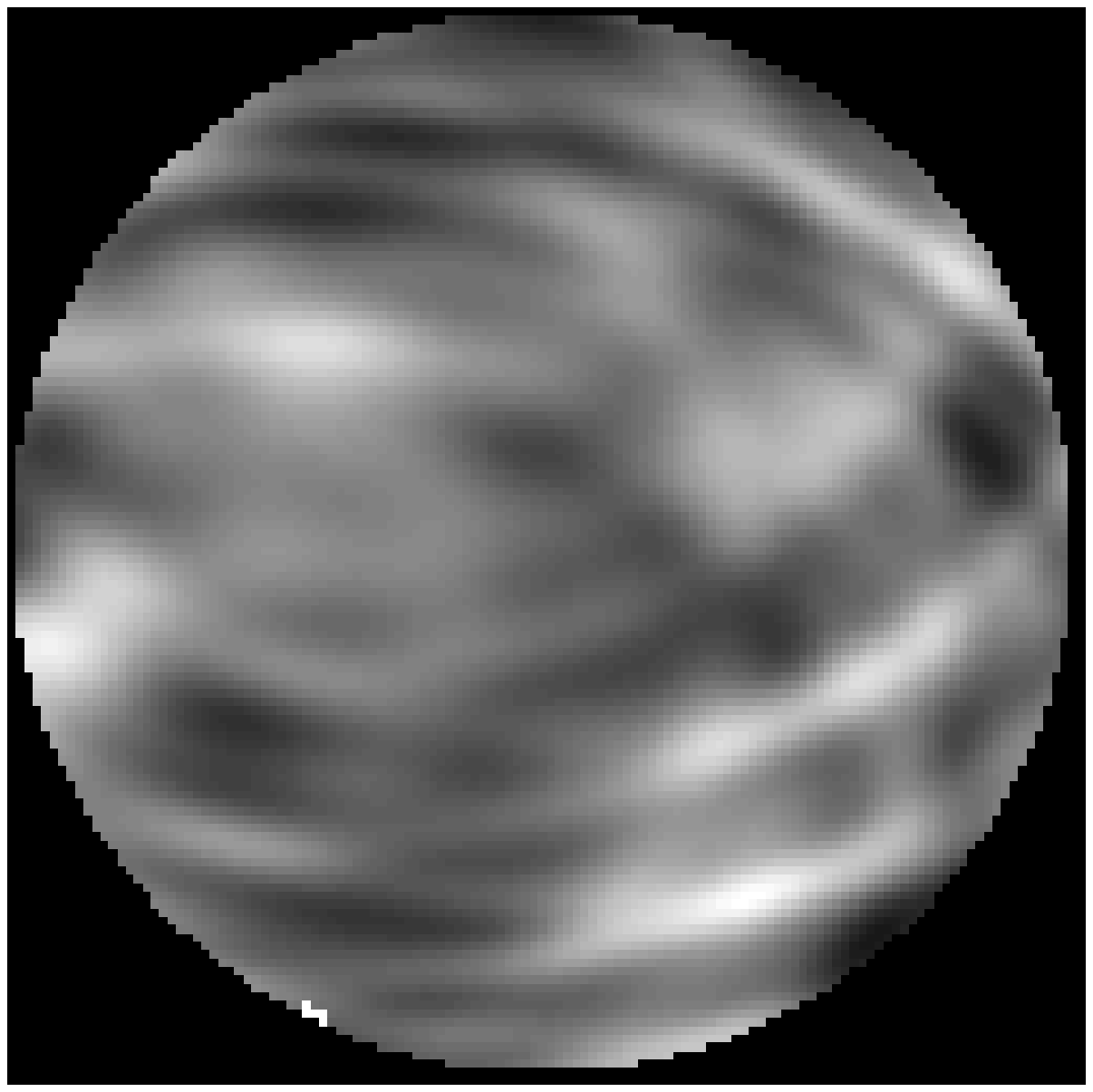}\\[-16.5ex]
   & \hspace{-0.16in}\includegraphics[width=.2\textwidth]{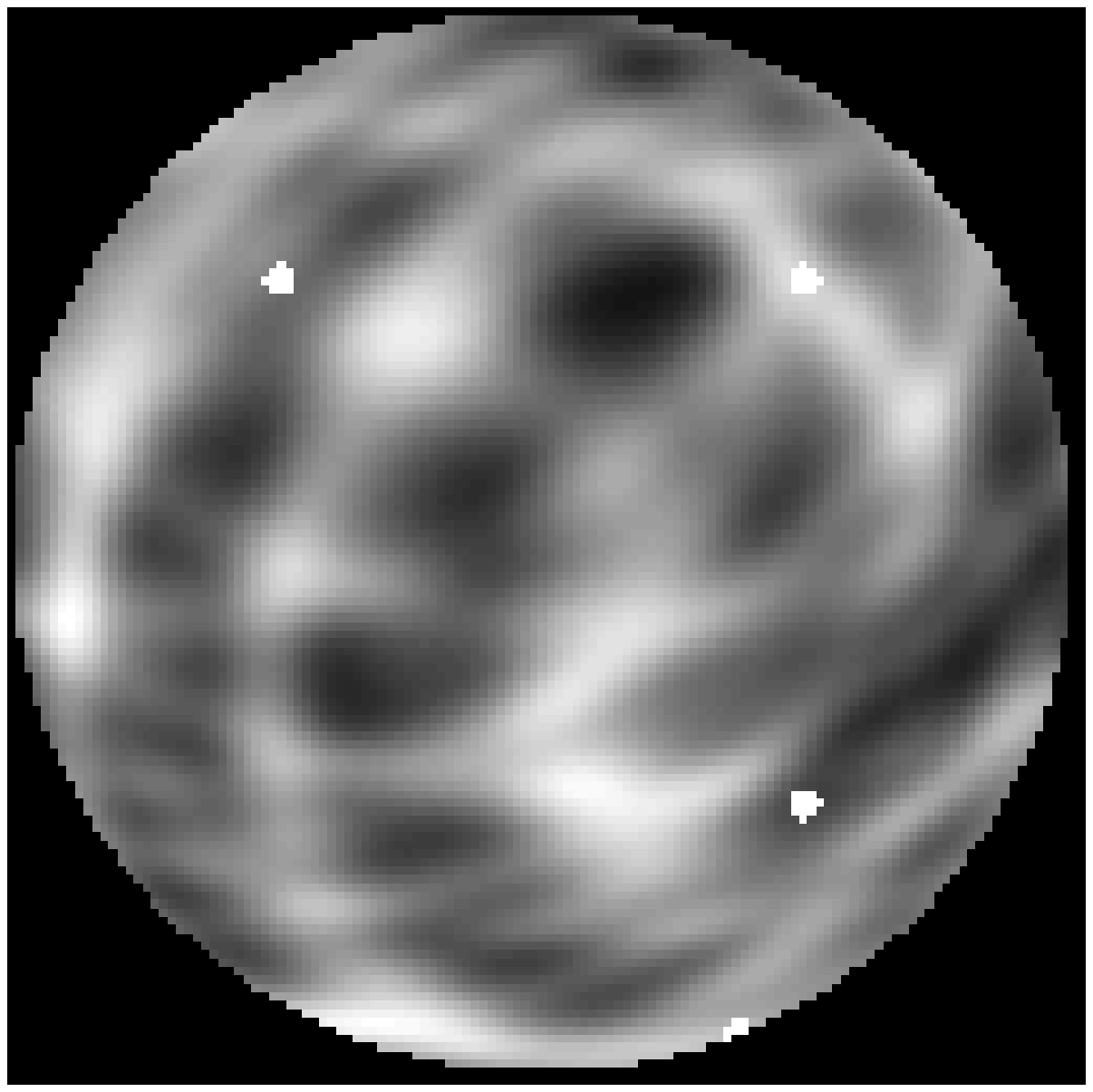}&
   \hspace{0.27in}\includegraphics[width=.2\textwidth]{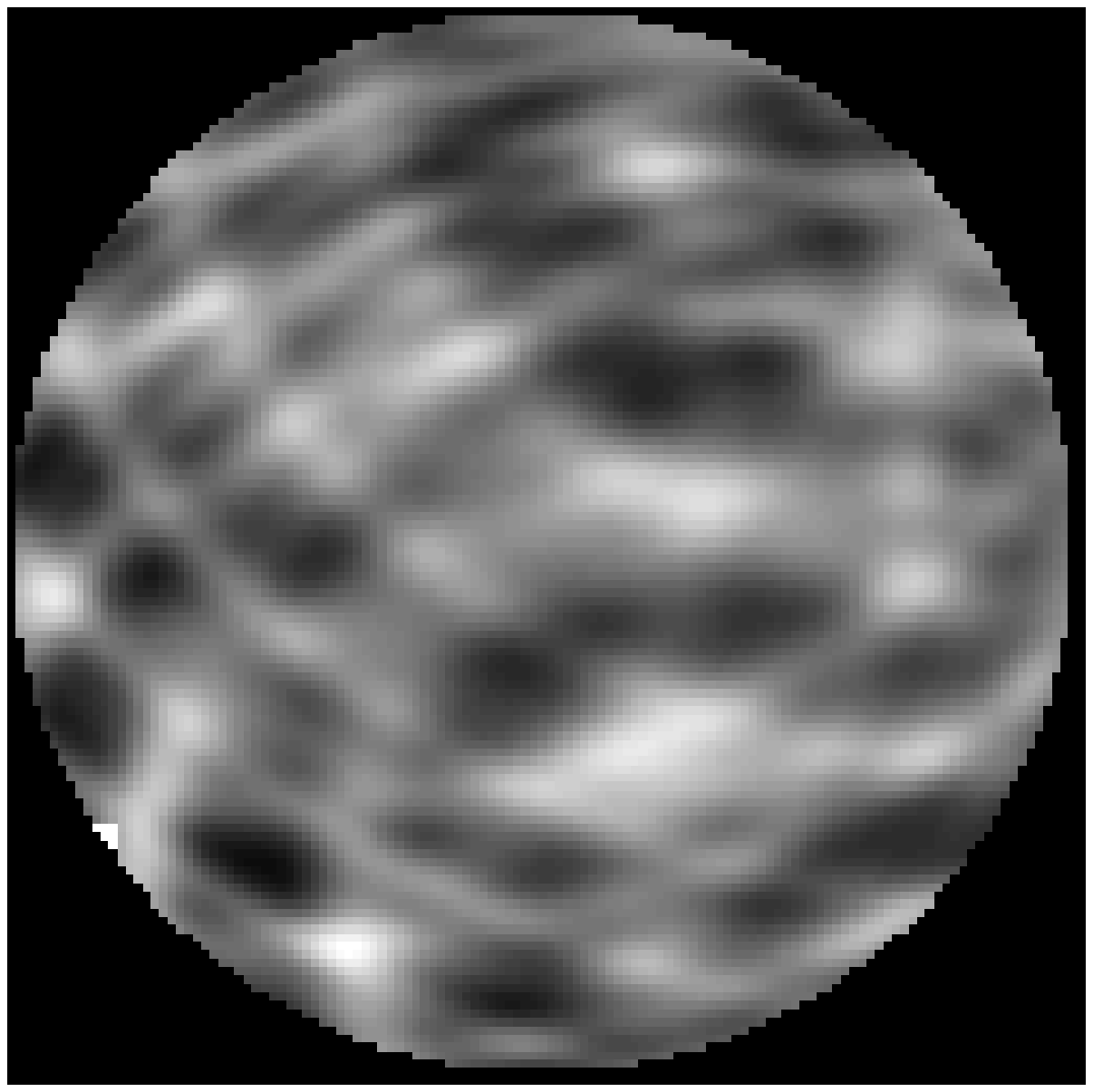}
 \end{tabular}
   \caption{Spherical tomography example.  On the left, the true image is provided, along with a few of the integration curves whose centers are located at $45^\circ$. Four sample images from the training dataset are provided on the right.}\label{fig:SphericalForwardModel}
   \end{center}
 \end{figure}
 
  Next, we assume that we have a dataset of training images for this problem consisting of $49$ images; four of the training images are provided in the right panel of Figure \ref{fig:SphericalForwardModel}.  All of the images contain a circular mask to denote the region of interest or region of visibility.  The inner regions of the images are generated using a linear combination of sine-squared functions, where the coefficients are random numbers uniformly distributed between $0.5$ and $1$, and the random numbers in sine-squared functions are uniformly distributed between $0$ and $128$.  Furthermore, each image is contaminated by at most $8$ ``freckles'' generated as white disks, where $5$ of them have radius $3$ and the rest have radius $4$. The freckles are randomly placed, where the origins of the freckles are uniformly distributed. Notice that the freckles do no appear in the true image.

 Given the training dataset $\{\bfs^{(1)}, \ldots, \bfs^{(49)}\}$, we first compute the (vectorized) mean image $\bar{\bfs}$ and the sample covariance matrix $\widehat{\bfQ}$. Next, assuming that the prior covariance matrix represents a Mat\'ern kernel, we solve optimization problem \eqref{ref:RademacherEst} to obtain ``learned'' Mat\'ern parameters $\check{\nu}$ and $\check{\ell}$ and consider the covariance matrix $\bfQ_{\text{learn}} = \bfQ(\check{\nu},\check{\ell}).$

\begin{figure}[b!]
    \centering
    \begin{tabular}{cc}
    \includegraphics[width=0.45\textwidth]{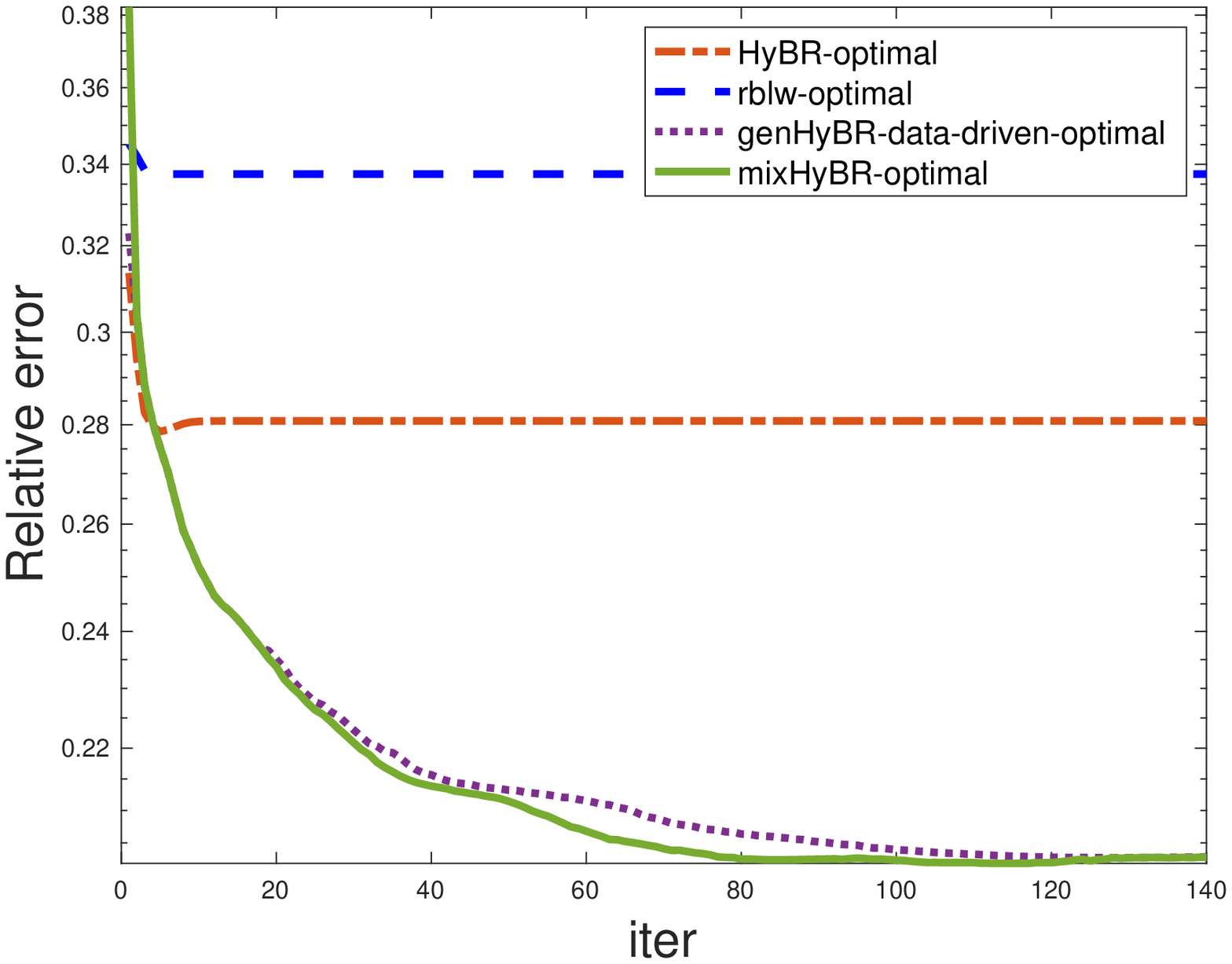} &
      \includegraphics[width=0.45\textwidth]{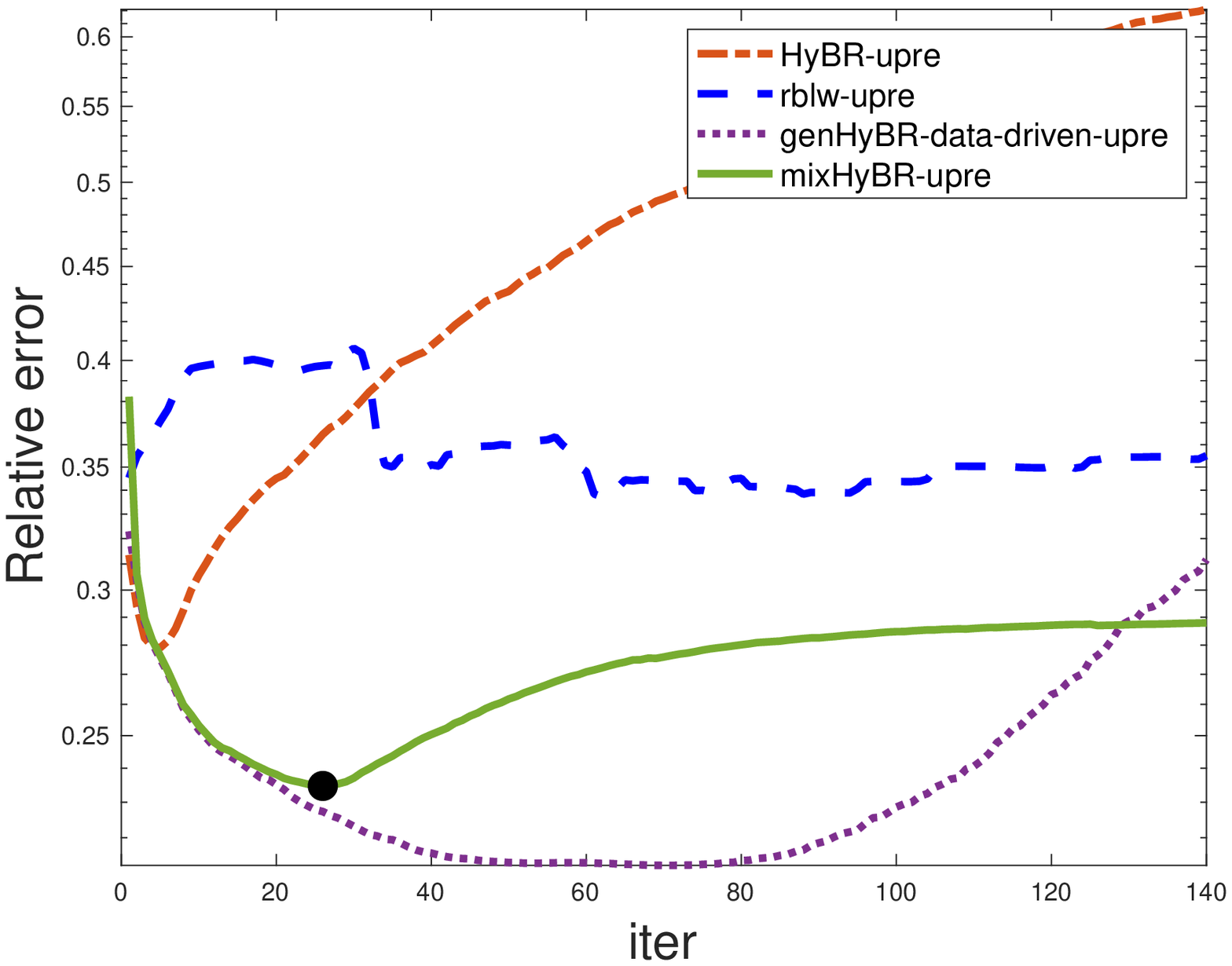} \\
      \includegraphics[width=0.45\textwidth]{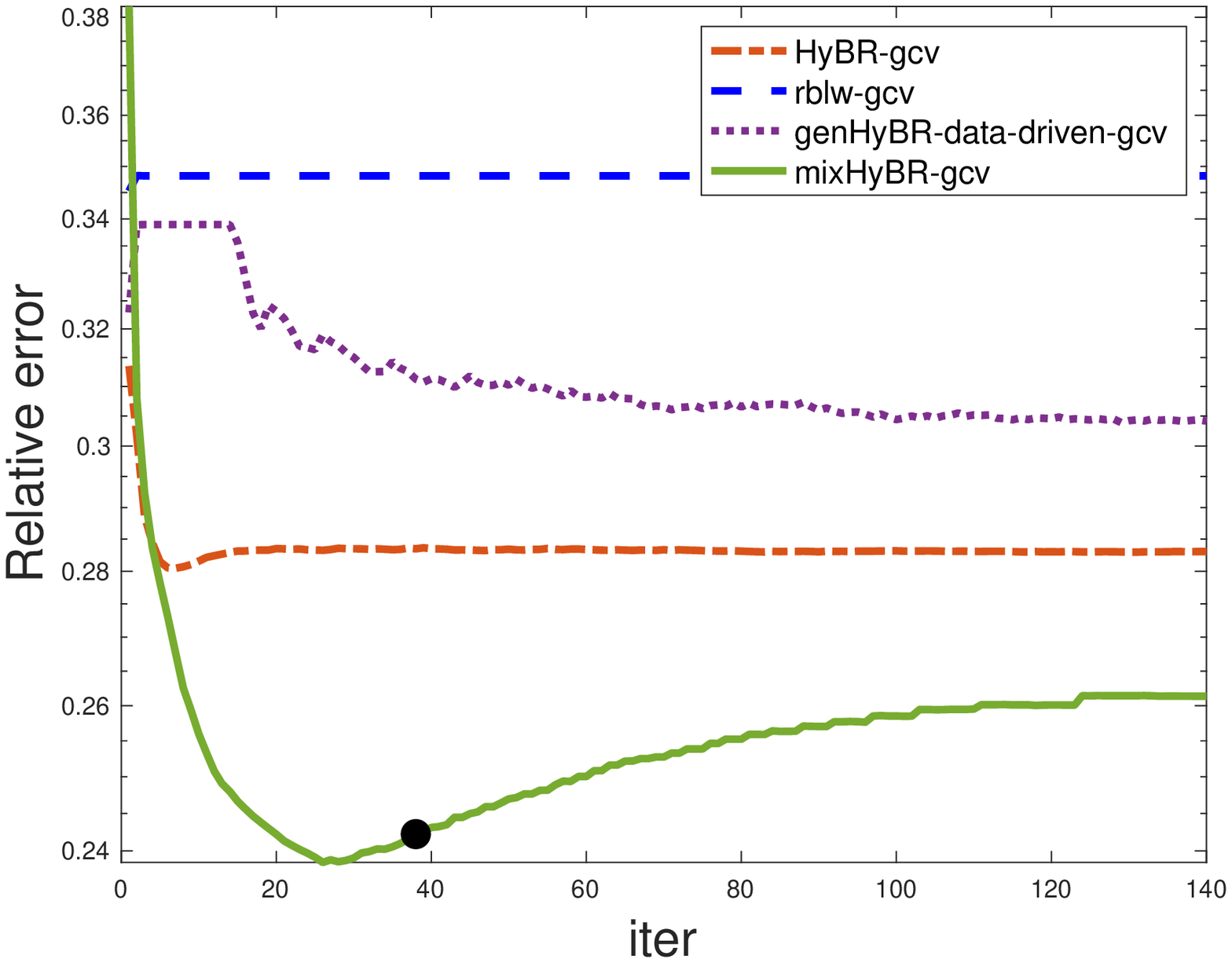}  &
      \includegraphics[width=0.45\textwidth]{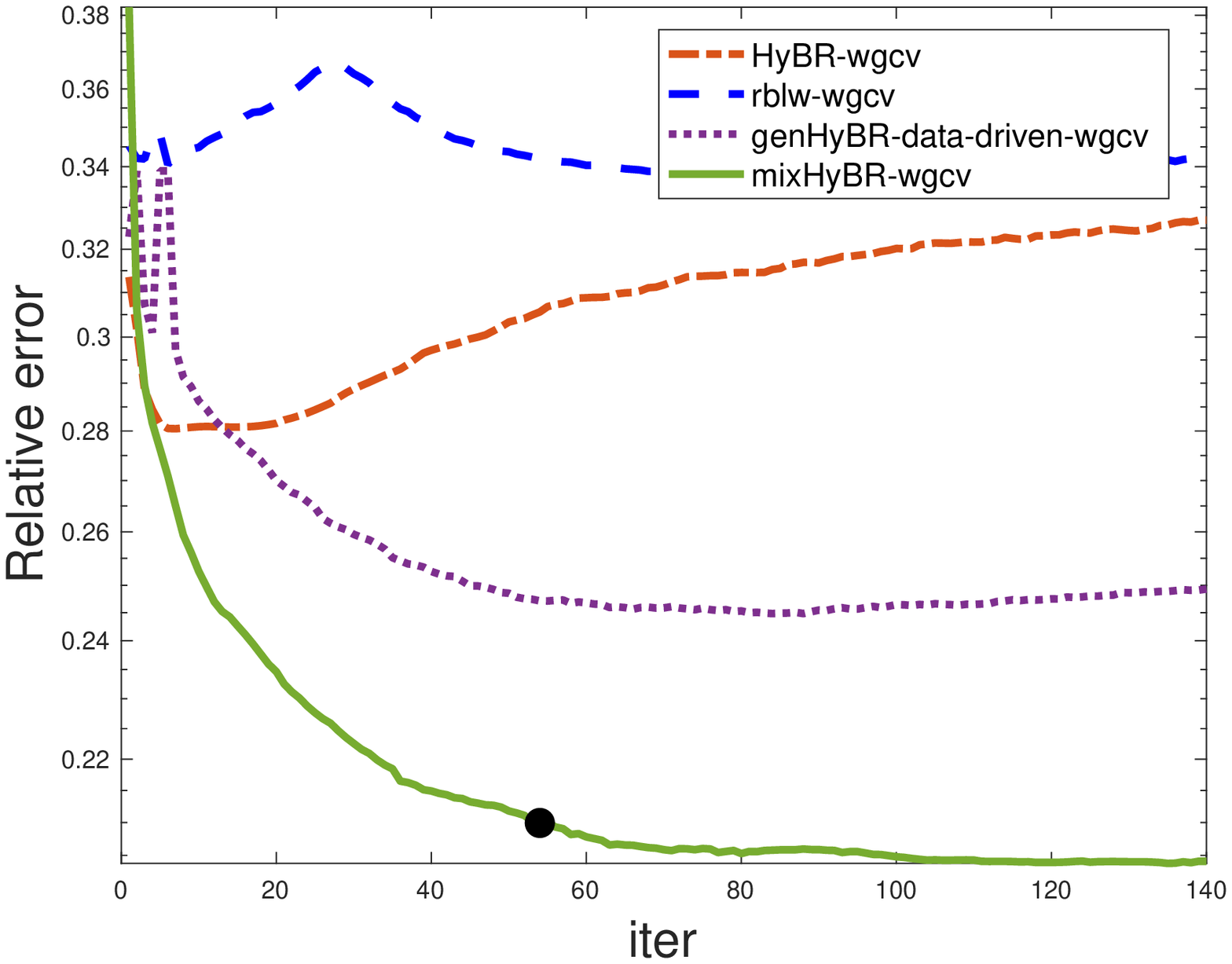}\\
      \end{tabular}
      \caption{Comparison of relative reconstruction error norms for various iterative hybrid approaches for spherical tomography reconstruction. The top left plot corresponds to using the optimal regularization parameters.  Other plots correspond to different methods to choose the regularization parameters, including UPRE, GCV, and WGCV.}
      \label{fig:other_com}
  \end{figure}
 
 We consider four hybrid iterative reconstruction methods, all with initial vector $\bar{\bfs}$. Given the training data, we run the genHyBR algorithm with $\bfQ = \bfQ_{\text{learn}}$ which we denote as `genHyBR-data-driven'.  We also provide results for `mixHyBR' where $\bfQ = \gamma\bfQ_{\text{learn}} + (1-\gamma)\widehat{\bfQ}$ where $\gamma$ and $\lambda$ are selected during the iterative process.
For comparison, we provide results for genHyBR with $\bfQ = \gamma\bfI + (1-\gamma)\widehat{\bfQ}$ where $\gamma$ was pre-selected using the Rao-Blackwell Ledoit and Wolf estimator (rblw) \cite{LW04, chen2009shrinkage,chen2010shrinkage}. We also provide results for HyBR where $\bfQ = \bfI$, but remark that this approach only uses the training data for the initial (sample mean) vector.  Note that for all considered methods, the regularization parameter $\lambda$ must be selected, and we investigate various approaches to do this.

 In Figure \ref{fig:other_com}, we provide relative reconstruction error norms computed as $\norm[2]{\bfs_k - \bfs_{\rm true}}/\norm[2]{\bfs_{\rm true}},$ where $\bfs_k$ is the reconstruction at the $k$th iteration. Each plot corresponds to a different method for selecting the regularization parameters.  For comparison, we provide in the top left plot results corresponding to the optimal regularization parameter, although these parameters cannot be computed in practice. 
 We observe that both genHyBR-data-driven and mixHyBR result in small error norms and that even with the optimal regularization parameter $\lambda$, the rblw approach performs poorly because of the poorly-estimated mixing parameter $\gamma.$
 We remark that we also compared these results to a shrinkage algorithm based on the oracle approximating shrinkage (OAS) estimator \cite{chen2009shrinkage,chen2010shrinkage} for obtaining $\gamma$.  However, we observed very similar results as rblw, so we do not include them here. 
 
For the automatic parameter selection methods, we observe that mixHyBR reconstructions with GCV and WGCV and genHyBR-data-driven reconstructions with UPRE have the smallest relative reconstruction error norms per iteration, compared to the other methods.  Thus, we observe that including a data-driven covariance matrix, if done properly, can be beneficial. The black dots denote the (automatically-selected) stopping iteration for mixHyBR. Although one may wish to tweak the stopping criteria, all of the examples with mixHyBR resulted in a good reconstruction with the described stopping criteria. For a better comparison of the different parameter selection methods, we provide all relative reconstruction errors for mixHyBR in Figure \ref{fig:compare_com}, where it is evident that relative errors for WGCV are very close to those for the optimal regularization parameter for this example. 
 
 \begin{figure}[t!]
    \centering
  \includegraphics[height=8cm, width=.8\textwidth]{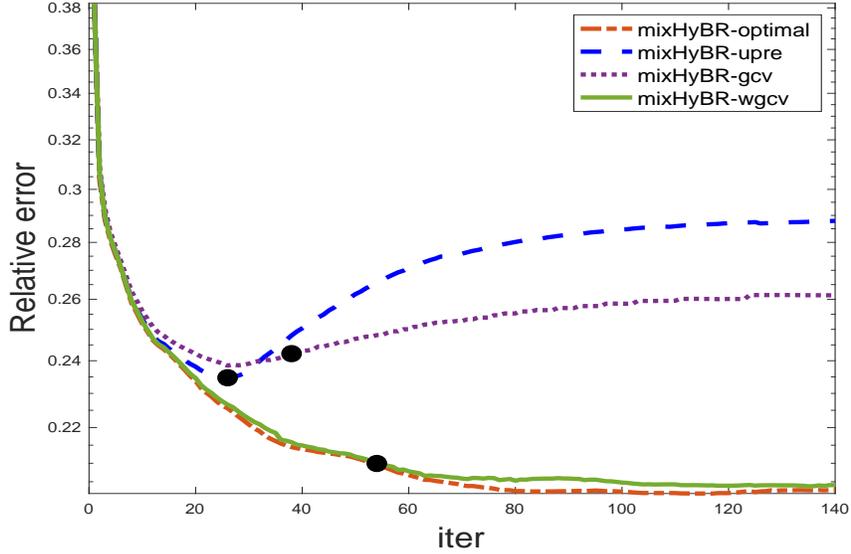}
      \caption{Relative reconstruction error norms per iteration of mixHyBR, for various regularization parameter choice methods.  Black dots denote the automatically computed stopping iteration.}
      \label{fig:compare_com}
  \end{figure}

  \begin{figure}[t!]
    \centering
      \includegraphics[width=1\textwidth]{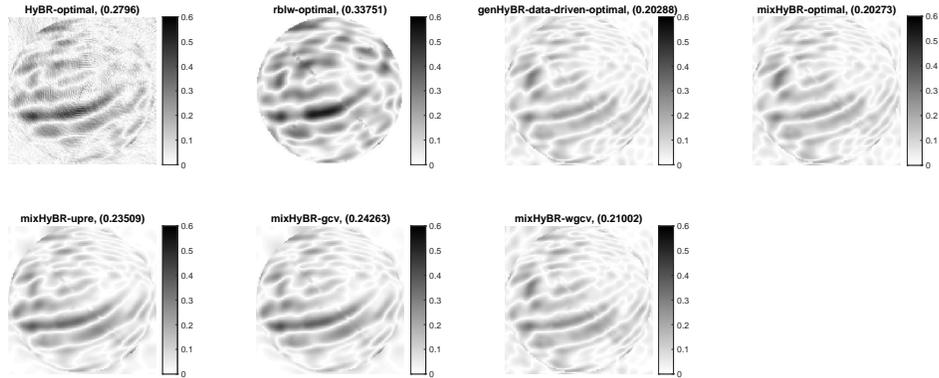} 
      \caption{Absolute error images (in inverted colormap), with relative reconstruction error norms provided in the titles.  The top row compares reconstructions using optimal regularization parameters, and the bottom row compares mixHyBR reconstructions with different parameter choice methods.}
      \label{fig:recon1}
  \end{figure}
 
Absolute error images, computed as $|\bfs_k - \bfs_\true|$, reshaped as an image, and displayed in inverted colormap, are provided in Figure \ref{fig:recon1}.  For better comparison, all error images have been put on the same scale, and dark regions corresponds to larger absolute errors.  Relative reconstruction error norms are provided in the titles.  In the top row, we compare reconstructions at iteration $140$ using the optimal regularization parameter.  Absolute error images in the bottom row correspond to mixHyBR reconstructions with automatic regularization parameter selection and correspond to the iteration determined by the stopping criteria. We notice that even with the optimal regularization parameter, the HyBR-optimal reconstruction suffers from the lack of sufficient prior information and the rblw-optimal reconstruction contains large errors due to the poor choice of $\gamma$ and disruptions due to freckles in the training data.
The mixHyBR and genHyBR-data-driven reconstructions have overall smaller absolute errors in the image.  For this example, all parameter selection methods combined with the stopping iteration performed reasonably well.

\subsection{Seismic tomography example} 
\label{sec:numericEx2}
In this experiment, we consider a linear inversion problem from crosswell tomography \cite{ambikasaran2013large}. Crosswell tomography is used to image the seismic wave speed in some region of interest, given data collected from multiple source-receiver pairs. The sources send out a seismic wave, and the receivers measure the travel time taken by the seismic wave to hit the receiver. The goal of the inverse problem is to image the slowness (reciprocal wave velocity) of the medium in the domain. We consider an example from Continuous Active Source Seismic Monitoring (CASSM) \cite{Daley2011}, where the goal is to monitor the spatial development of a small scale injection of CO$_2$ into a high quality reservoir. We consider reconstruction at a single time point and investigate the impact of including mixed Gaussian priors on the reconstruction.
 
 \begin{figure}[b!]
\begin{center}
\begin{tabular}{cc}
\includegraphics[width=0.4\textwidth]{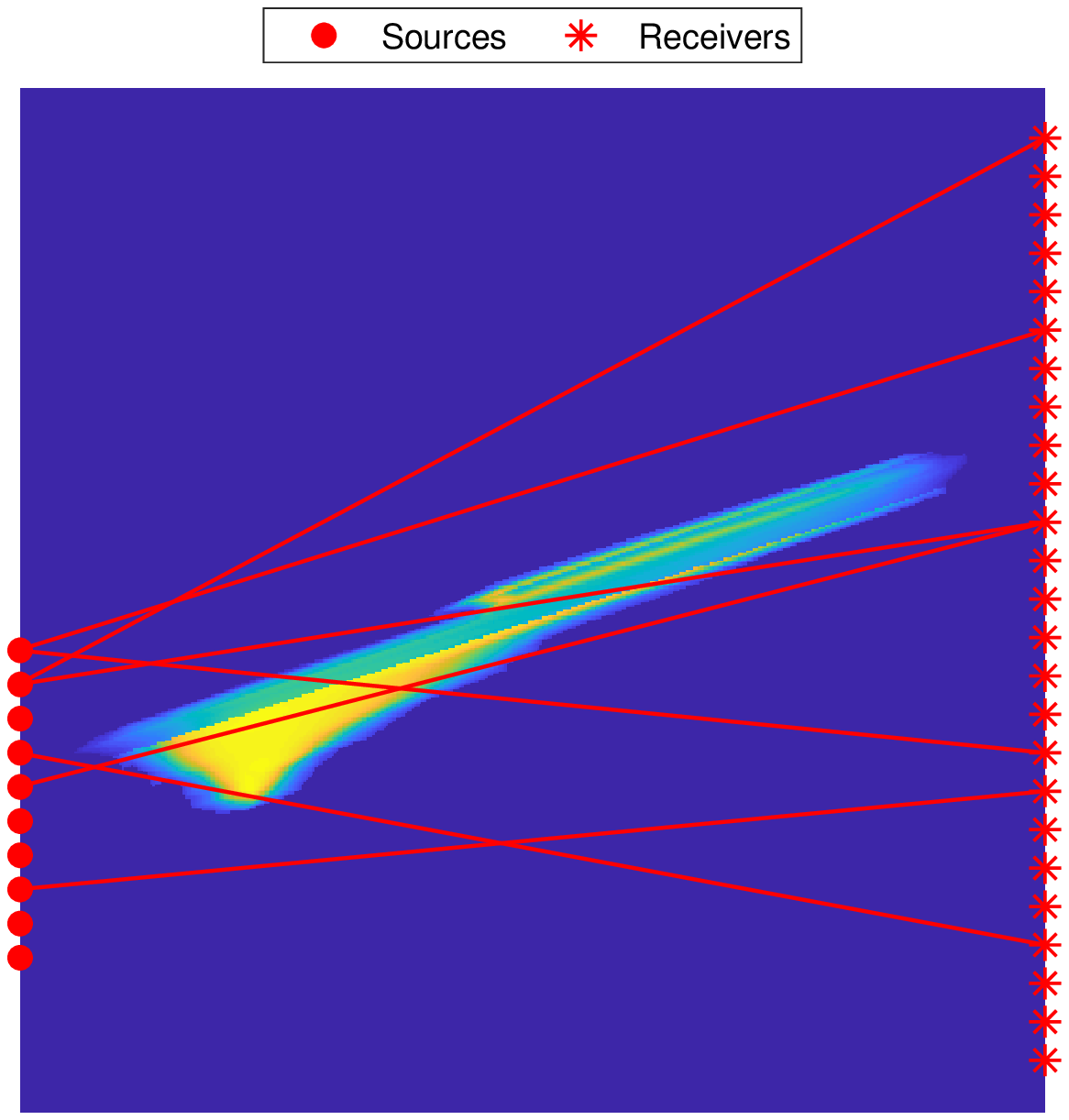} \quad & \quad 
\includegraphics[width=0.15\textwidth]{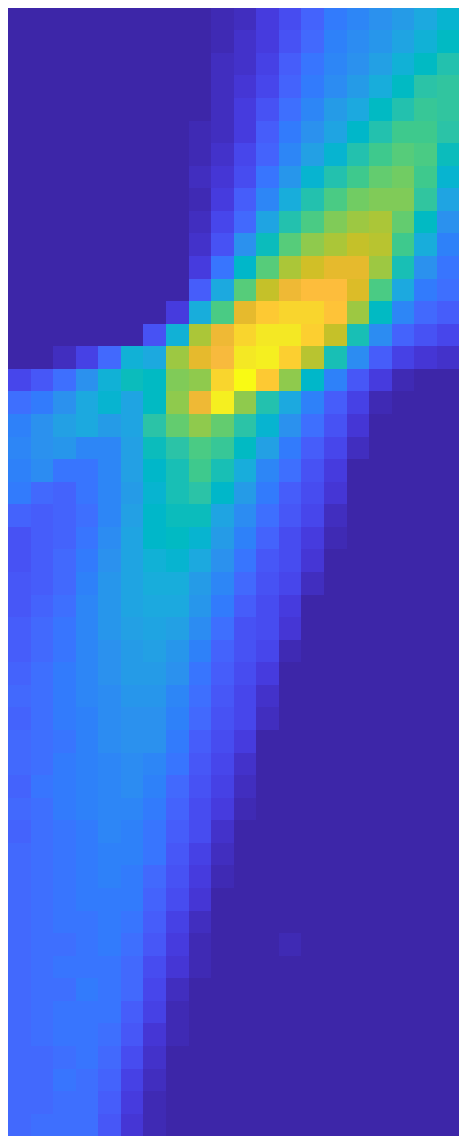}
\end{tabular}
  \end{center}
 \caption{CASSM example.  In the left panel, we  provide the true slowness field image, along with some of the locations of the sources and the detectors. Seven of the source-receiver pairs are highlighted in the figure.
 In the right panel, we provide the observations corresponding to $20$ sources and $50$ receivers.}
      \label{fig:CASSM_seismic_true}
\end{figure}

The inverse problem can be represented as (\ref{eq:problem}) where the goal is to reconstruct the slowness $\bfs \in \mathbb{R}^{n\times1}$ of the medium from the measured travel times $\bfd \in \mathbb{R}^{m\times1}$ which are assumed to be corrupted by Gaussian white noise $\bfepsilon \in \mathbb{R}^{m\times1}$.
In our problem setup, the true slowness field was discretized into $n=188,356$ cells, where the slowness within each cell is assumed to be constant. The true image (normalized between $0$ and $1$) is of size $434 \times 434$ and was obtained from \cite{cassmrawdata}.
For the observations, there were $m_s = 20$ sources and $m_r = 50$ receivers, so a total of $m = m_r m_s$ measurements.  Each row of the forward model matrix $\bfA \in \mathbb{R}^{m\times n}$ corresponds to a source-receiver pair. Since the wave travels along a straight line from source to receiver, only the cells lying on the straight line contribute to the non-zero entries. Hence, $\bfA$ is very sparse with $\mathcal{O}(\sqrt{m}n)$ non-zero entries.   The true image along with a schematic of the source-detector pairs are given in the left panel of Figure \ref{fig:CASSM_seismic_true}.  The observations, which contain $1\%$ noise, are provided in the right panel of Figure \ref{fig:CASSM_seismic_true}.

\begin{figure}[t!]
\begin{center}
 \includegraphics[height=6cm,width=.8\textwidth]{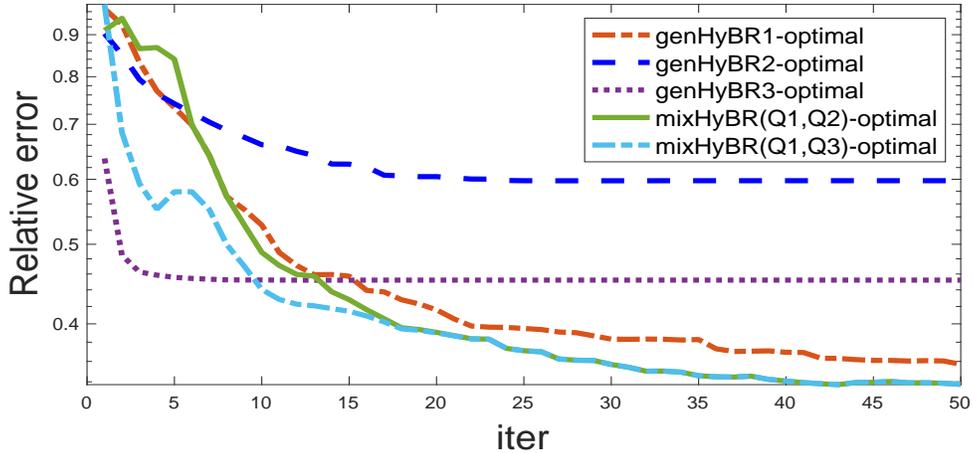}
\end{center}
 \caption{Comparison of relative reconstruction error norms for genHyBR and mixHyBR with optimal parameters $\gamma$ and $\ell$.}
      \label{fig:CASSM_opt}
\end{figure}
\begin{figure}[b!]
 \includegraphics[width=\textwidth]{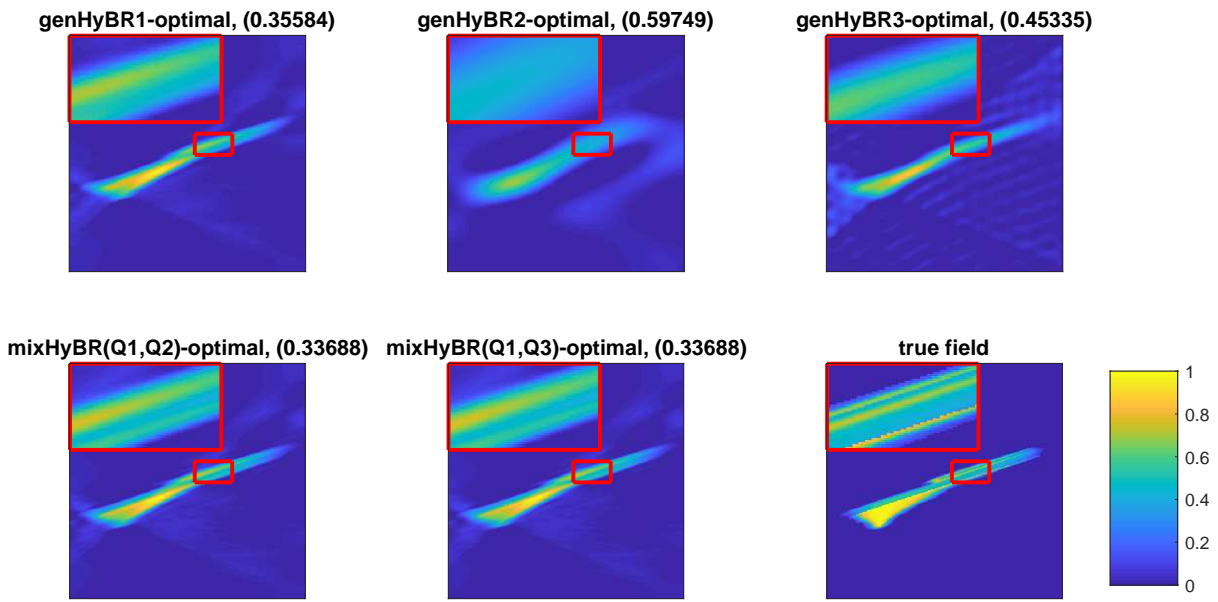}
 \caption{Reconstructions with zoomed subimages for CASSM example.  All of the reconstructions use the optimal regularization parameter and relative reconstruction errors are provided in the titles.}
      \label{fig:CASSM_seismic_recons_opt}
\end{figure}

Next we investigate the impact of different choices of $\bfQ$ on the reconstruction. First, we consider the genHyBR method with three different prior covariance matrices $\bfQ_1, \bfQ_2,$ and $\bfQ_3$ defined by a Mat\'{e}rn kernel with $\nu=0.5$ and $\ell=.25$, a rational quadratic with $\nu=2$ and $\ell=0.1$, and a sinc function with $\nu = 30\pi$, respectively. These approaches are denoted by `genHyBR1', `genHyBR2', and `genHyBR3' respectively.  Then we consider two mixHyBR approaches that include mixed Gaussian priors, where $\text{mixHyBR}(Q_1,Q_2)$ uses covariance matrix $\bfQ = \gamma\bfQ_1+(1-\gamma)\bfQ_2$ and
$\text{mixHyBR}(Q_1,Q_3)$ uses covariance matrix $\bfQ = \gamma\bfQ_1+(1-\gamma)\bfQ_3$, where the mixing parameter $\gamma$ is selected during the reconstruction process. For the optimally selected regularization parameters, we provide in Figure \ref{fig:CASSM_opt} the relative reconstruction error norms per iteration. 

We observe that if a good covariance matrix (in this case, $\bfQ_1$) is known in advance, stand-alone genHyBR can perform well and result in small relative reconstruction errors.  Otherwise, the relative reconstruction errors may remain large, and multiple solves with different covariance matrices would be needed to determine a good prior.  In this case, the mixHyBR approach can prove beneficial. The mixHyBR approaches produce reconstructions with overall smaller relative reconstruction errors than genHyBR with each covariance matrix alone. 
Image reconstructions, including a zoomed subregion, are provided in Figure \ref{fig:CASSM_seismic_recons_opt}.  Notice that the mixed Gaussian priors are better able to resolve some details of the true image.
Thus, incorporating mixed Gaussian priors can lead to improved reconstructions.

\begin{figure}[t!]
\begin{center}
 \begin{tabular}{cc} 
 \includegraphics[width=0.48\textwidth]{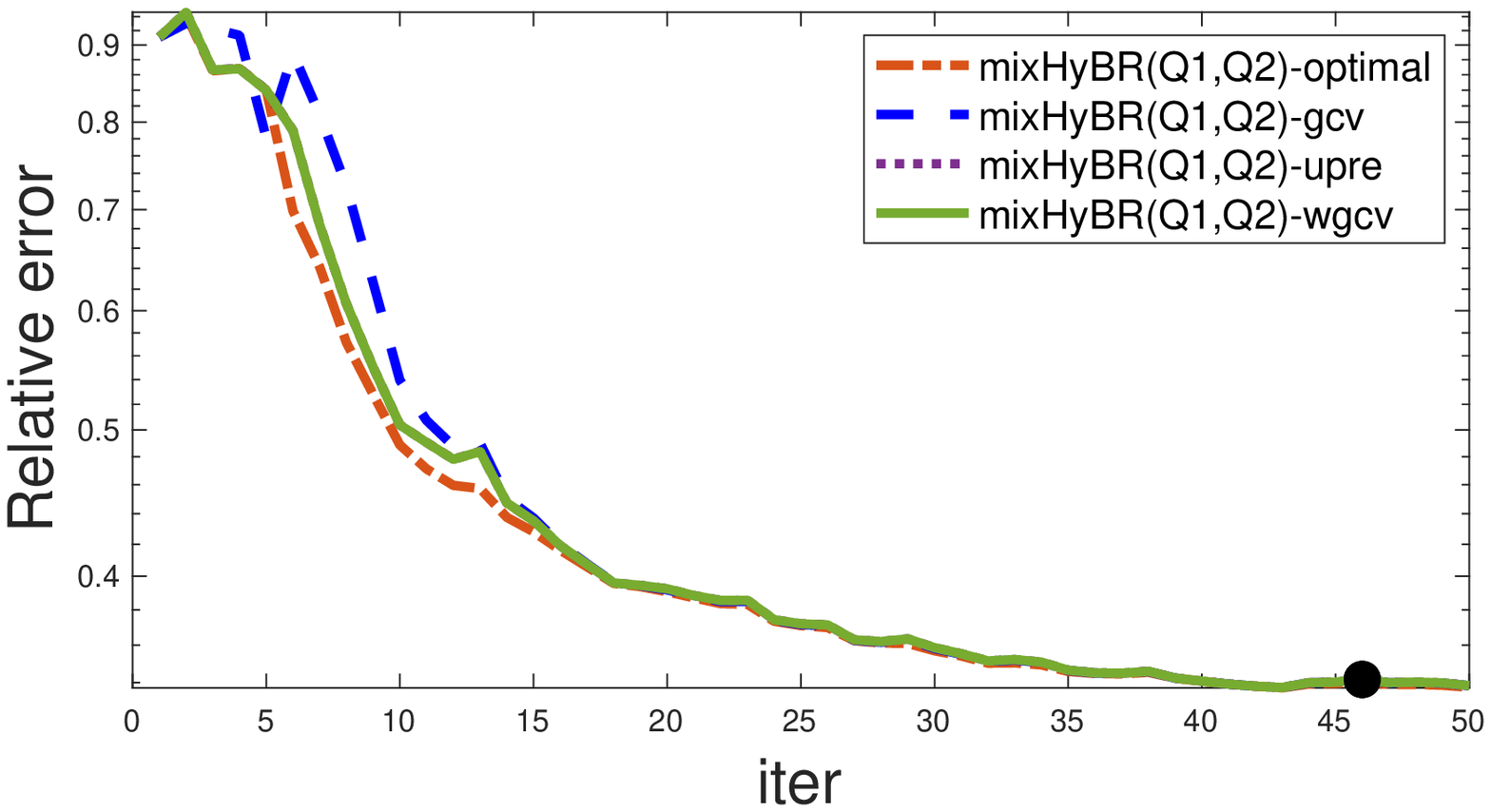}&
\includegraphics[width=0.48\textwidth]{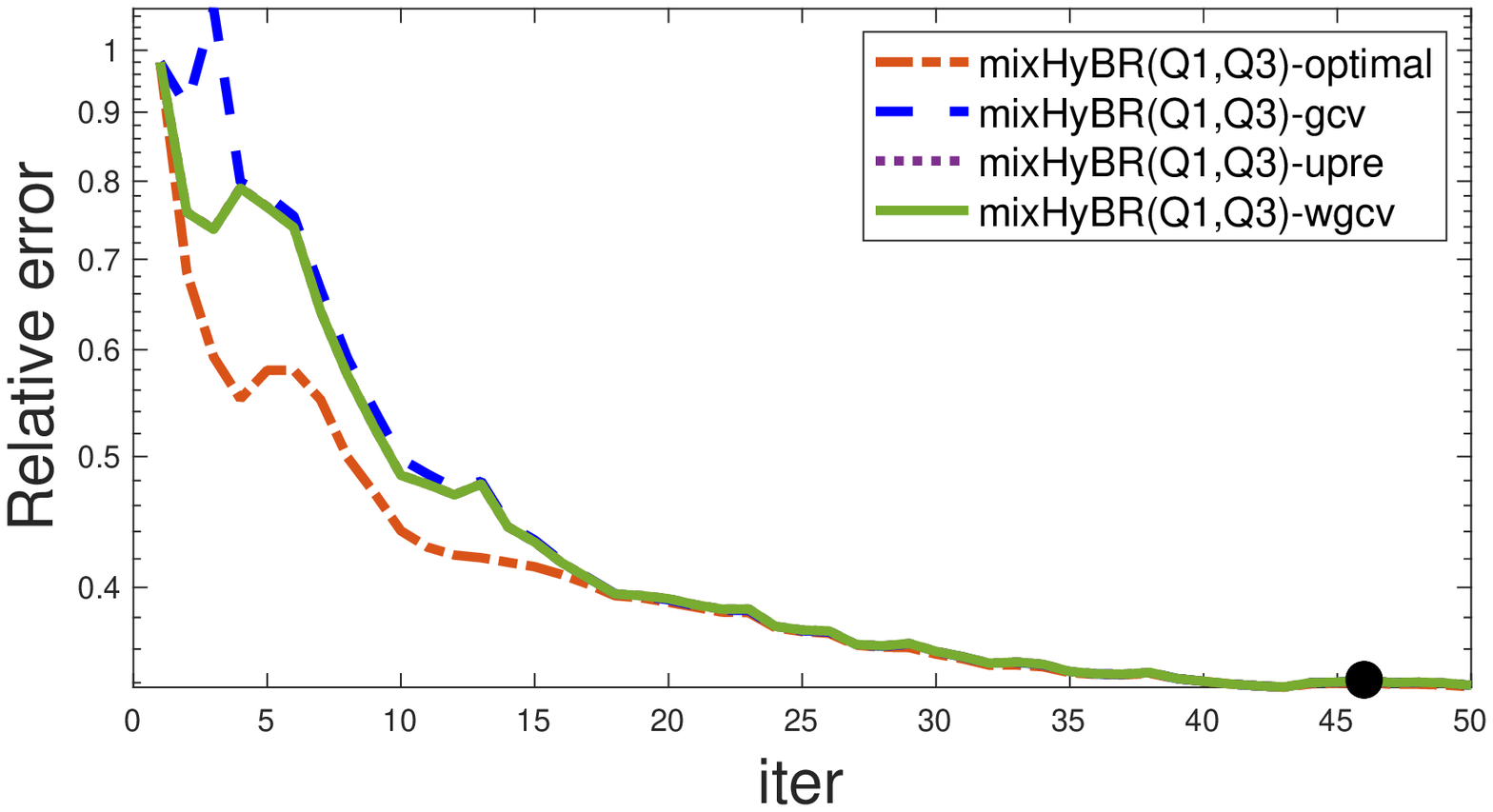}
\end{tabular}
\end{center}
 \caption{Comparison of relative reconstruction errors for $\text{mixHyBR}(Q_1,Q_2)$ (left) and $\text{mixHyBR}(Q_1,Q_3)$ (right) for different parameter choice methods. The automatically detected stopping iteration is marked with a black circle.}
      \label{fig:CASSM_seismic_true_forward_err}
\end{figure}

Next we investigate the performance of different regularization parameter selection methods within the mixHyBR methods.  Relative reconstruction errors for the GCV, WGCV, and UPRE methods with stopping iterates are provided in Figure \ref{fig:CASSM_seismic_true_forward_err}, along with results for the optimal parameters.  We used a tolerance of $10^{-6}$ for the residual errors.  We observe that all of the parameter selection methods work well for this example.

\section{Conclusions}
\label{sec:conclusions}
This paper describes a hybrid iterative projection method, dubbed mixHyBR, that is based on an extensions of the generalized Golub-Kahan bidiagonalization and that can be used for solving inverse problems (i.e., computing MAP estimates) with mixed Gaussian priors. The main advantage of this approach is that the mixing or blending parameter does not need to be known a priori, but rather can be estimated during the iterative process along with the regularization parameter. Various methods for selecting these parameters were considered and evaluated.  Furthermore, mixHyBR methods can easily incorporate data-driven priors where training data are used to define the prior covariance matrix itself (e.g., sample based priors) or to learn parameters for the covariance kernel function. Comparisons to widely-used shrinkage algorithms reveal that the mixed hybrid approaches are more robust under the presence of noise or artifacts in the data and enable greater flexibility when selecting suitable priors. Numerical results from both spherical and seismic tomography show the potential of these methods.

\appendix

\section{Proof of Theorem \ref{thm:convergence}}
\label{sec:appendix}
\begin{proof}

Based on (\ref{eq:proj-solution}) and (\ref{eq:D}),
 \begin{align}
\bfy_k(\lambda,\gamma) = \bfC_k(\gamma,\lambda)\begin{bmatrix}
    \beta_1\bfe_1 \\ \textbf{0}
\end{bmatrix}
 \end{align}
With $k=n$,  by (\ref{e_bk}), (\ref{e_vk}), (\ref{e_uk}), and (\ref{eq:traceDC}),
\begin{align*}
    \bfs_{n} &= {\bfmu} + \bfQ \bfV_n\bfy_n\\
    & = {\bfmu} + \bfQ\bfV_n\bfC_n(\gamma,\lambda)\begin{bmatrix}
    \beta_1\bfe_1 \\ \textbf{0}
\end{bmatrix}\\
    & = \bfmu + \bfQ\bfV_n(\bfV_n\t\bfQ\t\bfA\t\bfL_{\bfR}\t\bfL_{\bfR}\bfA\bfQ\bfV_n + \lambda^2\bfV_n\t\bfQ\bfV_n)^{-1}\bfV_n\t\bfQ\bfA\t\bfR^{-1}\bfb\\
    & = \bfmu + \bfQ(\bfQ\t\bfA\t\bfL_{\bfR}\t\bfL_{\bfR}\bfA\bfQ + \lambda^2\bfQ)^{-1}\bfQ\bfA\t\bfR^{-1}\bfb\\
    &= {\bfmu} + \bfQ(\bfA\t\bfR^{-1}\bfA\bfQ + \lambda^2\bfI_n)^{-1}\bfA\t\bfR^{-1}\bfb \\
    & = \bfs_{\rm MAP}.
\end{align*}

Therefore, the solution for \eqref{eq:projUpre} converges to the solution for \eqref{eq:fullUpre} and the solution for \eqref{eq:proj_gcv} converges to the solution for \eqref{eq:fullGCV} as $k$ increases.

\end{proof}

\section{Proof of Lemma \ref{lemma:res_proj}}
\label{sec:appendix2}
\begin{proof}
For the projected residual for $\bfx_k$, \begin{equation*}
    \begin{array}{rcl}
        \label{eq:convRes}
        \left\| \bfr_k^{\rm proj}(\gamma,\lambda) \right\|^2_2 & =& \left\| \bfD(\gamma)\bfy_k(\gamma,\lambda) - \begin{bmatrix} \beta_1\bfe_1 \\ \textbf{0} \end{bmatrix}\right\|^2_2 \\
        & = & \left\| (\gamma \bfL_{\bfR}\bfA\bfQ_1\bfV_k + (1-\gamma)\bfL_{\bfR}\bfA\bfQ_2\bfV_k)\bfy_k(\gamma,\lambda)-\bfL_{\bfR}\bfb \right\|^2_2 \\
        & = & \left\| \begin{bmatrix} \bfI_{k+1} & \widetilde{\bfU}_{k+1}\bfA\bfQ_2\bfV_k \\ \textbf{0} & \bfR_k \end{bmatrix}\begin{bmatrix}\gamma \bfB_k \\ (1-\gamma)\bfI_k \end{bmatrix}\bfy_k(\gamma,\lambda)-\begin{bmatrix} \beta_1\bfe_1 \\ \textbf{0} \end{bmatrix} \right\|^2_2 \\
        & = & \left\| \begin{bmatrix} \widetilde{\bfU}_{k+1} & \bfL_{\bfR}\bfA\bfQ_2\bfV_k \end{bmatrix}\begin{bmatrix} \gamma \bfB_k \\ (1-\gamma)\bfI_k \end{bmatrix}\bfy_k(\gamma,\lambda)-\bfL_{\bfR}\bfb \right\|^2_2 \\
        & = & \left\| (\gamma \widetilde{\bfU}_{k+1}\bfB_k + (1-\gamma)\bfL_{\bfR}\bfA\bfQ_2\bfV_k)\bfy_k(\gamma,\lambda)-\bfL_{\bfR}\bfb \right\|^2_2 \\
        & = & \left\| (\gamma \bfL_{\bfR}\bfA\bfQ_1\bfV_k + (1-\gamma)\bfL_{\bfR}\bfA\bfQ_2\bfV_k)\bfy_k(\gamma,\lambda)-\bfL_{\bfR}\bfb \right\|^2_2 \\
        & = & \left\|\bfL_{\bfR}\bfA\bfQ\bfx_k(\gamma,\lambda)-\bfL_{\bfR}\bfb\right\|^2_2 
    \end{array}
\end{equation*} and as $k\rightarrow n$, \begin{equation*}
    \left\|\bfL_{\bfR}\bfA\bfQ\bfx_k(\gamma,\lambda)-\bfL_{\bfR}\bfb\right\|^2_2 \rightarrow \left\|\bfL_{\bfR}\bfA\bfQ\bfx_n(\gamma,\lambda)-\bfL_{\bfR}\bfb\right\|^2_2.
\end{equation*} Since $\left\|\bfL_{\bfR}\bfA\bfQ\bfx_n(\gamma,\lambda)-\bfL_{\bfR}\bfb\right\|^2_2 = \left\|\bfr^{\rm full}(\gamma,\lambda)\right\|^2_2$, \begin{equation*}
    \left\| \bfr_k^{\rm proj}(\gamma,\lambda) \right\|^2_2 \rightarrow  \left\|\bfr^{\rm full}(\gamma,\lambda)\right\|^2_2  
\end{equation*} as $k\rightarrow n$.\\

\noindent For $k$th iteration in the projected problem \eqref{eq:proj-solution}, \begin{equation*}
    \begin{array}{rcl}
        \label{eq:DtD}
        \bfD_k(\gamma)\t\bfD_k(\gamma)& = & \begin{bmatrix} 
        \gamma\bfB_k + (1-\gamma)\widetilde{\bfU}_{k+1}\t\bfL_{\bfR}\bfA\bfQ_2\bfV_k \\ (1-\gamma)\bfR_k\end{bmatrix}\t\begin{bmatrix}  \gamma\bfB_k + (1-\gamma)\widetilde{\bfU}_{k+1}\t\bfL_{\bfR}\bfA\bfQ_2\bfV_k \\ (1-\gamma)\bfR_k
        \end{bmatrix}\\
        & = & \gamma^2\bfB_k\t\bfB_k + 2\gamma(1-\gamma)\bfB_k\t\bfU\t_{k+1}\bfL_{\bfR}\t\bfL_{\bfR}\bfA\bfQ_2\bfV_k \\
        & & + (1-\gamma)^2\bfV_k\t\bfQ_2\t\bfA\t\bfL_{\bfR}\t\bfL_{\bfR}\bfU_{k+1}\bfU_{k+1}\t\bfL_{\bfR}\t\bfL_{\bfR}\bfA\bfQ_2\bfV_k +(1-\gamma)^2\bfR_k\t\bfR_k \\
        & = & \gamma^2\bfV_k\t\bfQ_1\t\bfA\t\bfL_{\bfR}\t\bfL_{\bfR}\bfA\bfQ_1\bfV_k + 2\gamma(1-\gamma)  \bfV_k\t\bfQ_1\t\bfA\t  \bfL_{\bfR}\t\bfL_{\bfR}\bfA\bfQ_2\bfV_k \\
        & & +  (1-\gamma)^2\bfV_k\t\bfQ_2\t\bfA\t\bfL_{\bfR}\t\bfL_{\bfR}\bfU_{k+1}\bfU_{k+1}\t\bfL_{\bfR}\t\bfL_{\bfR}\bfA\bfQ_2\bfV_k \\
        & & +(1-\gamma)^2\bfV_k\t\bfQ_2\t\bfA\t\bfL_{\bfR}\t( \bfI_{k+1}-
        \bfL_{\bfR}\bfU_{k+1}\bfU\t_{k+1}\bfL_{\bfR}\t)\bfL_{\bfR}\bfA\bfQ_2\bfV_k \\
        & = & \gamma^2\bfV_k\t\bfQ_1\t\bfA\t\bfL_{\bfR}\t\bfL_{\bfR}\bfA\bfQ_1\bfV_k + 2\gamma(1-\gamma) \bfV_k\t\bfQ_1\t\bfA\t\bfL_{\bfR}\t\bfL_{\bfR}\bfA\bfQ_2\bfV_k \\
        & & +  (1-\gamma)^2\bfV_k\t\bfQ_2\t\bfA\t\bfL_{\bfR}\t  \bfL_{\bfR}\bfA\bfQ_2\bfV_k \\
        & = & (\bfL_{\bfR}\bfA\bfQ\bfV_k)\t  \bfL_{\bfR}\bfA\bfQ\bfV_k
    \end{array}
\end{equation*} Therefore, \begin{equation}
    \begin{array}{rcl}
        \label{eq:traceDC}
         {\rm tr}(\bfD_k(\gamma)\bfC_k(\gamma,\lambda)) & = &  {\rm tr}(\bfD_k(\gamma)(\bfD_k(\gamma)\t\bfD_k(\gamma)+\lambda^2\gamma\bfI_k + \lambda^2(1-\gamma)\bfV_k\t\bfQ_2\bfV_k)^{-1}\bfD_k(\gamma)\t ) \\
         & = & {\rm tr}((\bfD_k(\gamma)\t\bfD_k(\gamma)+\lambda^2\gamma\bfI_k + \lambda^2(1-\gamma)\bfV_k\t\bfQ_2\bfV_k)^{-1}\bfD_k(\gamma)\t \bfD_k(\gamma)) \\
         & = & {\rm tr}(((\bfL_{\bfR}\bfA\bfQ\bfV_k)\t\bfL_{\bfR}\bfA\bfQ\bfV_k + \lambda^2 \bfV_k\t \bfQ\bfV_k)^{-1}(\bfL_{\bfR}\bfA\bfQ\bfV_k)\t(\bfL_{\bfR}\bfA\bfQ\bfV_k))\\
         & = & {\rm tr}((\bfL_{\bfR}\bfA\bfQ\bfV_k)((\bfL_{\bfR}\bfA\bfQ\bfV_k)\t\bfL_{\bfR}\bfA\bfQ\bfV_k + \lambda^2 \bfV_k\t \bfQ\bfV_k)^{-1}(\bfL_{\bfR}\bfA\bfQ\bfV_k)\t) \\
         & \rightarrow & {\rm tr}((\bfL_{\bfR}\bfA\bfQ\bfV_n)((\bfL_{\bfR}\bfA\bfQ\bfV_n)\t\bfL_{\bfR}\bfA\bfQ\bfV_n + \lambda^2 \bfV_n\t \bfQ\bfV_n)^{-1}(\bfL_{\bfR}\bfA\bfQ\bfV_n)\t) \\
         & = & {\rm tr}((\bfL_{\bfR}\bfA\bfQ\bfV_n)\bfV_n^{-1}((\bfL_{\bfR}\bfA\bfQ)\t\bfL_{\bfR}\bfA\bfQ + \lambda^2 \bfQ)^{-1}\bfV_n^{-\top}(\bfL_{\bfR}\bfA\bfQ\bfV_n)\t) \\
         & = & {\rm tr}((\bfL_{\bfR}\bfA\bfQ) ((\bfL_{\bfR}\bfA\bfQ)\t\bfL_{\bfR}\bfA\bfQ + \lambda^2 \bfQ)^{-1}  (\bfL_{\bfR}\bfA\bfQ)\t) \\
         & = & {\rm tr}(A(\gamma,\lambda))
    \end{array}
\end{equation} with the invertible $\bfV_n$ since $\bfV_n\t\bfQ_1 \bfV_n = \bfI_n$ and $\bfV_n\in\R^{n \times n}$ is square matrix.
\end{proof}

\section*{Acknowledgments}
This work was partially supported by NSF DMS 1723005 and NSF DMS 1654175. J. Chung would also like to acknowledge support from the Alexander von Humboldt Foundation. 

\bibliographystyle{siamplain}
\bibliography{data_genHyBR.bib}

\end{document}